\documentclass{amsart}
\usepackage{amssymb}
\usepackage[all]{xy}
\usepackage{graphicx}
\usepackage{upref,amsxtra,amscd}
\usepackage{amsthm}
\usepackage{amsmath}
\usepackage{amsfonts}

\date{\today}

\newcommand{\Z}{{\mathbb Z}}
\newcommand{\R}{{\mathbb R}}
\newcommand{\C}{{\mathbb C}}
\newcommand{\N}{{\mathbb N}}
\newcommand{\T}{{\mathbb T}}

\newtheorem{theorem}{Theorem}[section]
\newtheorem{remark}[theorem]{Remark}
\newtheorem{lemma}[theorem]{Lemma}
\newtheorem{prop}[theorem]{Proposition}
\newtheorem{coro}[theorem]{Corollary}

\def\N{{\mathbb N}}

\theoremstyle{remark}
\newtheorem{rem}[theorem]{Remark}
\theoremstyle{definition}

\newcommand{\abs}[1]{\left\vert#1\right\vert}
\newcommand{\set}[1]{\left\{#1\right\}}
\newcommand{\eqdef}{\overset{\mathrm{def}}=}

\usepackage{enumerate}
\usepackage{subfigure}

\begin{document}

\title[The Fibonacci Hamiltonian]{The Fibonacci Hamiltonian}

\author[D.\ Damanik]{David Damanik}

\address{Department of Mathematics, Rice University, Houston, TX~77005, USA}

\email{damanik@rice.edu}

\thanks{D.\ D.\ was supported in part by a Simons Fellowship and NSF grant DMS--1067988.}

\author[A.\ Gorodetski]{Anton Gorodetski}

\address{Department of Mathematics, University of California, Irvine, CA~92697, USA}

\email{asgor@math.uci.edu}

\thanks{A.\ G.\ was supported in part by NSF grants DMS-1301515 and 
IIS-1018433.}

\author[W.\ Yessen]{William Yessen}

\address{Department of Mathematics, Rice University, Houston, TX~77005, USA}

\email{yessen@rice.edu}

\thanks{W.\ Y.\ was supported by the NSF Mathematical Sciences Postdoctoral Research Fellowship, DMS-1304287}

\begin{abstract}
We consider the Fibonacci Hamiltonian, the central model in the study of electronic properties of one-dimensional quasicrystals, and provide a detailed description of its spectrum and spectral characteristics (namely, the optimal H\"older exponent of the integrated density of states, the dimension of the density of states measure, the dimension of the spectrum, and the upper transport exponent) for all values of the coupling constant (in contrast to all previous quantitative results, which could be established only in the regime of small or large coupling).

In particular, we show that the spectrum of this operator is a dynamically defined Cantor set and that the density of states measure is exact-dimensional; this implies that all standard fractal dimensions coincide in each case. We show that all the gaps of the spectrum allowed by the gap labeling theorem are open for all values of the coupling constant. Also, we establish strict inequalities between the four spectral characteristics in question, and provide the exact large coupling asymptotics of the dimension of the density of states measure (for the other three quantities, the large coupling asymptotics were known before).

A crucial ingredient of the paper is the relation between spectral properties of the Fibonacci Hamiltonian and dynamical properties of the Fibonacci trace map (such as dimensional characteristics of the non-wandering hyperbolic set and its measure of maximal entropy as well as other equilibrium measures, topological entropy, multipliers of periodic orbits). We establish exact identities relating the spectral and dynamical quantities, and show the connection between the spectral quantities and the thermodynamic pressure function.
\end{abstract}

\maketitle

\tableofcontents

\section{Introduction}\label{s.intro}

In this paper we study the Fibonacci Hamiltonian. Along with the almost Mathieu operator, this particular operator is the most heavily studied Schr\"odinger operator, with dozens of mathematics papers and hundreds of physics papers devoted to it. There are several reasons for this extensive interest in the spectral properties of the Fibonacci Hamiltonian. The first and perhaps most important reason is that this operator is a central model in mathematical physics. Namely it is relevant in the study of electronic properties of quasicrystals. Quasicrystals are materials that were first discovered by Shechtman in 1982, and this discovery led to a paradigm shift in materials science. In diffraction experiments they produce patterns consisting of sharp bright spots, the so-called Bragg peaks, while at the same time these diffraction patterns display symmetries that conclusively prove that the arrangement of atoms in the sample cannot be periodic. This came as a surprise as it had been believed that Bragg peaks can only be observed in the diffraction of materials for which the arrangement of atoms is periodic. It therefore took the scientific community a while until this discovery was properly digested and accepted, and it was finally published in the 1984 paper \cite{SBGC84}. Shechtman has received numerous honors and distinctions for this discovery, including the 2011 Nobel Prize in Chemistry.

Since the 1980's mathematicians have studied appropriate models of quasicrystals. Naturally, the choice of these models is guided by the distinctive property of real-life quasicrystals, namely that of having a pure point diffraction which in turn displays symmetries that rule out periodicity. The central examples of the commonly accepted mathematical quasicrystal models are the Fibonacci tilings or sequences in one dimension and the Penrose tilings in two dimensions. Indeed, these examples belong to all classes of models that are typically considered in their respective dimension. In particular, they may be generated both by inflation and by a cut-and-project scheme.

Mathematical quasicrystal models are studied from many perspectives, including dynamical systems, harmonic analysis, spectral theory, discrete geometry, combinatorics, and algebra; compare \cite{BG13, BM00, M97}. The study of electronic or quantum transport in quasicrystals, which is the perspective we take in this paper, naturally leads to the consideration of Schr\"odinger operators with potentials modeling a quasicrystalline environment. Choosing the environment given by the Fibonacci tiling or sequence, this leads to the discrete one-dimensional Schr\"odinger operator
\begin{equation}\label{e.fib}
[H_{\lambda,\omega} u](n) = u(n+1) + u(n-1) + \lambda \chi_{[1-\alpha,1)}(n \alpha + \omega \!\!\!\! \mod 1) u(n),
\end{equation}
acting in $\ell^2(\Z)$, where $\lambda > 0$ is the coupling constant, $\alpha = \frac{\sqrt{5}-1}{2}$ is the frequency, and $\omega \in \T = \R / \Z$ is the phase. In particular, $\alpha$ is the inverse of the golden ratio
\begin{equation}\label{e.goldenratio}
\varphi = \frac{\sqrt{5} + 1}{2}.
\end{equation}
Alternatively, the potential can be generated by the Fibonacci substitution $a \mapsto ab$, $b \mapsto a$; compare \cite{D00, D07, DEG}. The operator family \eqref{e.fib} is what is usually called the Fibonacci Hamiltonian. It was proposed and initially studied by Kohmoto et al.\ \cite{KKT} and by Ostlund et al.\ \cite{OPRSS}, prior to the publication of \cite{SBGC84}, as a quasi-periodic model that can be solved exactly by renormalization group techniques. The relevance to quasicrystals was only established and discussed later. The first papers on the model in the mathematics literature belong to Casdagli \cite{Cas} and S\"ut\H{o} \cite{S87}.

The second reason for the extensive interest in this operator family is that it has exciting spectral properties. Independently of the relevance of the model to physics, the Fibonacci Hamiltonian also serves as a paradigm for many spectral phenomena that had been considered exotic prior to the 1980's. For example, it persistently displays Cantor spectrum, zero-measure spectrum, purely singular continuous spectral measures, and anomalous transport. That these properties are even possible for a physically relevant model was surprising prior to discovery, and the fact that these properties can be rigorously established only adds to the importance of the model. Specifically, it is often difficult to answer questions about the spectrum and the spectral type for a given Schr\"odinger operator with an aperiodic and non-decaying potential (periodic and decaying potentials are well understood; compare, e.g., \cite{SimonSzego} and \cite{DK07}). In many cases one rather resorts to statements about members in families of operators that hold generically or with probability one. The operator families corresponding to the Fibonacci and almost Mathieu cases are special in that quite detailed and difficult questions about these operators can be answered for all members of the respective family. Establishing this has been the objective of many publications in the past three decades focusing on these two operator families; see, for example, the surveys \cite{D00, D07, D15, DEG, J07, JM14}.

In this paper we show that the spectrum of the Fibonacci Hamiltonian is a dynamically defined Cantor set and that the density of states measure is exact-dimensional; this implies that all standard fractal dimensions coincide in each case. We show that all the gaps of the spectrum allowed by the gap labeling theorem are open for all values of the coupling constant. Also, we consider the optimal H\"older exponent of the integrated density of states, the dimension of the density of states measure, the dimension of the spectrum, and the upper transport exponent, establish strict inequalities between them, and provide the exact large coupling asymptotics of the dimension of the density of states measure (for the other three quantities, the large coupling asymptotics were known before). We also provide the explicit relations between these spectral characteristics and the dynamical properties of the Fibonacci trace map (such as dimensional characteristics of the non-wandering hyperbolic set and its measure of maximal entropy as well as other equilibrium measures, topological entropy, multipliers of periodic orbits). We establish exact identities relating the spectral and dynamical quantities, and show the connection between the spectral quantities and the thermodynamic pressure function. Our results not just improve but complete our understanding of many spectral characteristics and properties of Fibonacci Hamiltonian. In the rest of the introduction we provide the exact statement of the results and discuss them in the context of previously known facts.

\subsection{The Spectrum of the Fibonacci Hamiltonian}

The spectrum of the Fibonacci Hamiltonian $H_{\lambda,\omega}$ is independent of $\omega$ and may therefore by denoted by $\Sigma_\lambda$. This follows from strong operator convergence and the minimality of an irrational rotation of the circle. It was shown by S\"ut\H{o} in \cite{S89} that $\Sigma_\lambda$ is a Cantor set of zero Lebesgue measure for every $\lambda > 0$. The zero-measure property in turn rules out any absolutely continuous spectrum for $H_{\lambda,\omega}$. Complementing this, Damanik and Lenz showed that $H_{\lambda,\omega}$ has no eigenvalues \cite{DL99a}, and hence for all parameter values, all spectral measures are purely singular continuous. This answers the basic qualitative spectral questions about this operator family.

Our first result shows that actually $\Sigma_\lambda$ is a {\it dynamically defined}\footnote{Sometimes the term ``regular'' is also used.} Cantor set, that is, it belongs to a special and heavily studied class of Cantor sets that have strong self-similarity properties (see \cite{Palis1993} for the formal definition and a detailed discussion of the properties of dynamically defined Cantor sets).

\begin{theorem}\label{t.ddcs}
For every $\lambda > 0$, $\Sigma_\lambda$ is a dynamically defined Cantor set. In particular, for every $E \in \Sigma_\lambda$ and every $\varepsilon > 0$, we have
$$
\dim_H \left( (E-\varepsilon, E+\varepsilon) \cap \Sigma_\lambda \right) = \dim_B \left( (E-\varepsilon, E+\varepsilon) \cap \Sigma_\lambda \right) = \dim_H \Sigma_\lambda = \dim_B \Sigma_\lambda.
$$
\end{theorem}

Here, $\dim_H(S)$ (resp., $\dim_B(S)$) denotes the Hausdorff (resp., box counting) dimension of the Borel set $S \subset \R$. Stating the identities above contains the implicit assertion that the box counting dimension of the set in question exists.

This result was previously known for $\lambda \ge 16$ \cite{Cas} and $\lambda > 0$ sufficiently small \cite{DG09}. Knowing that the spectrum is a dynamically defined Cantor set not only establishes the equality of all standard fractal dimensions of the set (and shows that this common dimension is bounded away from zero and one), it also serves as the starting point for further studies. For example, higher-dimensional separable models may be considered and their spectra turn out to be given by the sum of the one-dimensional spectra; compare, for example, \cite{DG11, DGS}. This leads to a study of sums of dynamically defined Cantor sets, which is an extensively investigated problem about which much is known (see, e.g., \cite{HS, MY} and references therein).

\subsection{Transport Exponents}

Given that the operator $H_{\lambda,\omega}$ has purely singular continuous spectrum for all parameter values, the RAGE Theorem (see, e.g., \cite[Theorem~XI.115]{RS3}) suggests that when studying the Schr\"odinger time evolution for this Schr\"odinger operator, that is, $e^{-itH_{\lambda,\omega}} \psi$ for some initial state $\psi \in \ell^2(\Z)$, one should consider time-averaged quantities. For simplicity, let us consider initial states of the form $\delta_n$, $n \in \Z$. Since a translation in space simply results in an adjustment of the phase, we may without loss of generality focus on the particular case $\psi = \delta_0$. The time-averaged spreading of $e^{-itH_{\lambda,\omega}} \delta_0$ is usually captured on a power-law scale as follows; compare, for example, \cite{DT10, L96}. For $p > 0$, consider the $p$-th moment of the position operator,
$$
\langle |X|_{\delta_0}^p \rangle (t) = \sum_{n \in \Z} |n|^p | \langle e^{-itH_{\lambda,\omega}} \delta_0 , \delta_n \rangle |^2
$$
We average in time as follows. If $f(t)$ is a function of $t > 0$ and $T > 0$ is given, we denote the time-averaged function at $T$ by $\langle f \rangle (T)$:
$$
\langle f \rangle (T) = \frac{2}{T} \int_0^{\infty} e^{-2t/T} f(t) \, dt.
$$
Then, the corresponding upper and lower transport exponents $\tilde \beta^+_{\delta_0}(p)$ and $\tilde \beta^-_{\delta_0}(p)$ are given, respectively, by
$$
\tilde \beta^+_{\delta_0}(p) = \limsup_{T \to \infty} \frac{\log \langle \langle |X|_{\delta_0}^p \rangle \rangle (T) }{p \, \log T},
$$
$$
\tilde \beta^-_{\delta_0}(p) = \liminf_{T \to \infty} \frac{\log \langle \langle |X|_{\delta_0}^p \rangle \rangle (T) }{p \, \log T}.
$$
The transport exponents $\tilde \beta^\pm_{\delta_0}(p)$ belong to $[0,1]$ and are non-decreasing in $p$ (see, e.g., \cite{DT10}), and hence the following limits exist:
\begin{align*}
\tilde \alpha_l^\pm & = \lim_{p \to 0} \tilde \beta^\pm_{\delta_0}(p), \\
\tilde \alpha_u^\pm & = \lim_{p \to \infty} \tilde \beta^\pm_{\delta_0}(p).
\end{align*}

Ballistic transport corresponds to transport exponents being equal to one, diffusive transport corresponds to the value $\frac12$, and vanishing transport exponents correspond to (some weak form of) dynamical localization. In all other cases, transport is called anomalous. The Fibonacci Hamiltonian has long been the primary candidate for a model exhibiting anomalous transport, going back at least to the work of Abe and Hiramoto \cite{AH}. Many papers have been devoted to a study of the transport properties of the Fibonacci Hamiltonian; see, for example, \cite{BLS, D98, D05, DKL, DST, DT03, DT05, DT07, DT08, JL2, kkl}. For example, it is known that all the transport exponents defined above are strictly positive for all $\lambda > 0$, $\omega \in \T$; see \cite{DKL}. On the other hand, upper bounds for all the transport exponents were shown in \cite{DT07} for $\lambda > 8$ (see also \cite{BLS} for a somewhat weaker result). The exact large coupling asymptotics of $\tilde \alpha_u^\pm$ were identified in \cite{DT08}, where is was shown that
\begin{equation}\label{e.inftyapproach2}
\lim_{\lambda \to \infty} \tilde \alpha_u^\pm \cdot \log \lambda = 2 \log \varphi,
\end{equation}
uniformly in $\omega \in \T$. In particular, the Fibonacci Hamiltonian indeed gives rise to anomalous transport for sufficiently large coupling. The behavior in the weak coupling regime was studied in \cite{DG14}, where it was shown that there is a constant $c > 0$ such that for $\lambda > 0$ sufficiently small, we have
$$
1 - c\lambda^2 \le \tilde \alpha_u^\pm \le 1,
$$
uniformly in $\omega \in \T$.

While it is of clear interest to identify the asymptotic behavior of $\tilde \alpha_u^\pm$ in the large and small coupling regimes, and in particular show that the asymptotic behavior of $\tilde \alpha_u^+$ coincides with that of $\tilde \alpha_u^-$, the following questions remain. What can we say for a given value of $\lambda$? Can we for example give an explicit expression for $\tilde \alpha_u^+$ or $\tilde \alpha_u^-$? Can we even show that $\tilde \alpha_u^+$ and $\tilde \alpha_u^-$ coincide for the given value of $\lambda$ (and $\omega$)?

We will address these questions in this paper. An explicit description of $\tilde \alpha_u^\pm$ will be given in Theorem~\ref{t.identities} stated later in this introduction (it will require the trace map formalism, which will be recalled in Subsection~\ref{ss.tracemap}). A particular consequence of the description given there is that the desired identity holds:

\begin{theorem}\label{t.equaltransportexponents}
For every $\lambda > 0$, $\tilde \alpha^+_u(\lambda)$ and $\tilde \alpha^-_u(\lambda)$ are equal and independent of $\omega \in \T$.
\end{theorem}

The interpretation of this statement is that, for all values of the coupling constant and the phase, the fastest part of the wavepacket travels uniformly on a power-law scale. That is, there aren't two different sequences of time scales along which the ``front of the wavepacket'' moves at two different power-law rates. To the best of our knowledge this is the first time this phenomenon has been rigorously established for a model for which $\tilde \alpha^+_u(\lambda)$ and $\tilde \alpha^-_u(\lambda)$ take fractional values. The reason for this is that it is usually very difficult to identify transport exponents exactly (if they take fractional values) and hence most results only establish estimates for them.

\medskip

\subsection{Density of States Measure and Gap Labeling}

Let us recall the definition of the density of states measure and some derived quantities. By the spectral theorem, there are Borel probability measures $\mu_{\lambda,\omega}$ on $\R$ such that
$$
\langle \delta_0 , g(H_{\lambda,\omega}) \delta_0 \rangle = \int g(E) \, d\mu_{\lambda,\omega}(E)
$$
for all bounded measurable functions $g$. The \emph{density of states measure} $\nu_\lambda$ is given by the $\omega$-average of these measures with respect to Lebesgue measure, that is,
$$
\int_\T \langle \delta_0 , g(H_{\lambda,\omega}) \delta_0 \rangle \, d\omega = \int g(E) \, d\nu_\lambda(E)
$$
for all bounded measurable functions $g$. By general principles, the density of states measure is non-atomic and its topological support is $\Sigma_\lambda$. The fact that $\Sigma_\lambda$ has zero Lebesgue measure therefore implies that $\nu_\lambda$ is singular continuous for every $\lambda > 0$.
The density of states measure can also be obtained by counting the number of eigenvalues per unit volume, in a given energy region, of restrictions of the operator to finite intervals (which explains the terminology). Indeed, for any real $a < b$,
$$
\nu_\lambda(a,b) = \lim_{L \to \infty} \frac{1}{L} \# \big\{ \text{eigenvalues of } H_{\lambda,\omega}|_{[1,L]} \text{ that lie in } (a,b) \big\},
$$
uniformly in $\omega$; compare \cite{H}. Here, for definiteness, $H_{\lambda,\omega}|_{[1,L]}$ is defined with Dirichlet boundary conditions.

We will be interested in the optimal H\"older exponent $\gamma_\lambda$ of $\nu_\lambda$. That is, $\gamma_\lambda$ is the unique number in $[0,1]$ such that the following two properties hold.

\begin{enumerate}

\item For any $\gamma < \gamma_\lambda$ and any sufficiently small interval $I \subset \mathbb{R}$, we have $\nu(I) < |I|^{\gamma}$;

\item For any $\tilde \gamma > \gamma_\lambda$ and any $\varepsilon > 0$, there exists an interval $I \subset \mathbb{R}$ such that $|I| < \varepsilon$ and $\nu(I) > |I|^{\tilde \gamma}$.

\end{enumerate}

The optimal H\"older exponent of the density of states measure is studied for other popular discrete Schr\"odinger operators in numerous papers; see, for example, \cite{AJ10, B00, GS01, GS08, H09} and references therein. For the Fibonacci case in the regime of small or large coupling, it was studied in \cite{DG13}. In particular, it was shown that $\gamma_\lambda\to 1/2$ as $\lambda \to 0$ and $\gamma_\lambda\to 0$ as $\lambda \to \infty$ (the explicit rate at which it does so is recalled in Theorem~\ref{t.asymptotics} below). In all these works only estimates and asymptotics for the optimal H\"older exponent were established. In this paper, we will express the optimal H\"older exponent in the Fibonacci case explicitly, for any value of the coupling constant, in terms of dynamical quantities related to the Fibonacci trace map (the explicit formula is provided in Theorem \ref{t.identities} below). For small values of the coupling constant $\lambda$, this will allow us to give an exact formula for $\gamma_\lambda$ as a function of $\lambda$; see Corollary~\ref{c.holderexp}.

The distribution function of the density of states measure is called the integrated density of states and denoted by $N_\lambda$. Thus, for $E \in \R$, we have
$$
N_\lambda(E) = \int \chi_{(-\infty,E]} \, d\nu_\lambda = \lim_{L \to \infty} \frac{1}{L} \# \big\{ \text{eigenvalues of } H_{\lambda,\omega}|_{[1,L]} \text{ that are } \le E \big\},
$$
uniformly in $\omega$.

Since $\Sigma_\lambda$ is the topological support of $\nu_\lambda$, it follows that $N_\lambda$ is constant on each gap of $\Sigma_\lambda$, where any connected component of $\R \setminus \Sigma_\lambda$ is called a gap of $\Sigma_\lambda$. This value may be used as the label of the gap in question. The gap labeling theorem (see, e.g., \cite{B92, J86}) provides a set that is defined purely in terms of the underlying dynamical system generating the ergodic family of potentials in question (in our case this is either the irrational rotation of the circle by the (inverse of the) golden ratio, or the shift transformation on the subshift generated by the Fibonacci substitution), to which all gap labels must belong. This general gap labeling theorem specializes in the Fibonacci case to the following statement (see, e.g., \cite[Eq.~(6.7)]{BBG92}):
\begin{equation}\label{f.fibgaplabels}
\{ N_\lambda(E) : E \in \R \setminus \Sigma_\lambda \} \subseteq \{ \{ m \varphi \} : m \in \Z \} \cup \{ 1 \}
\end{equation}
for every $\lambda > 0$. Here $\{ m \varphi \}$ denotes the fractional part of $m \varphi$, that is, $\{ m \varphi \} = m \varphi - \lfloor m \varphi \rfloor$. Notice that the set of possible gap labels is indeed $\lambda$-independent and only depends on the value of $\varphi$ from the underlying circle rotation. Since $\varphi$ is irrational, the set of gap labels is dense.

In general, a dense set of possible gap labels is indicative of a Cantor spectrum and hence a common (and attractive) stronger version of proving Cantor spectrum is to show that the operator ``has all its gaps open.'' For example, the Ten Martini Problem for the almost Mathieu operator is to show Cantor spectrum, while the Dry Ten Martini Problem is to show that all labels correspond to gaps in the spectrum. The former problem has been completely solved \cite{AJ}, while the latter has not yet been completely settled (it remains open for the case of critical coupling and non-Liouville frequency; see \cite{AJ, AYZ14, CEY90} and references therein). Indeed, it is in general a hard problem to show that all labels given by the gap labeling theorem correspond to gaps and there are only few results of this kind.

Here we show the stronger (or ``dry'') form of Cantor spectrum for the Fibonacci Hamiltonian and establish complete gap labeling:

\begin{theorem}\label{t.completegaplabeling}
For every $\lambda > 0$, all gaps allowed by the gap labeling theorem are open. That is,
\begin{equation}\label{f.completelabeling}
\{ N_\lambda(E) : E \in \R \setminus \Sigma_\lambda \} = \{ \{ m \varphi \} : m \in \Z \} \cup \{ 1 \}.
\end{equation}
\end{theorem}

Raymond proved \eqref{f.completelabeling} for $\lambda > 4$ \cite{Ra} and Damanik and Gorodetski proved \eqref{f.completelabeling} for $\lambda > 0$ sufficiently small \cite{DG11}. In \cite{DG11} it was also shown that all gaps open linearly as the coupling constant is turned on. It was conjectured in \cite{DG11} that \eqref{f.completelabeling} holds for every $\lambda > 0$, and Theorem~\ref{t.completegaplabeling} proves this conjecture.

Our next result concerns the exact-dimensionality of the density of states measure for every value of the coupling constant.

\begin{theorem}\label{t.doesexactdim}
For every $\lambda > 0$, the density of states measure $\nu_\lambda$ is exact-dimensional. Namely, for every $\lambda > 0$, the limit {\rm (}called the scaling exponent of $\nu_\lambda$ at $E${\rm )}
$$
\lim_{\varepsilon \downarrow 0} \frac{\log \nu_\lambda(E - \varepsilon , E + \varepsilon)}{\log \varepsilon}
$$
$\nu_\lambda$-almost everywhere exists  and is constant.
\end{theorem}

In \cite{DG12} Damanik and Gorodetski had shown the exact-dimensionality of $\nu_\lambda$ for $\lambda > 0$ sufficiently small. A particular consequence of the exact-dimensionality of $\nu_\lambda$ is that virtually all the known characteristics of dimension type of the measure coincide. In particular, the following four dimensions associated with the measure $\nu_\lambda$, those most relevant to quantum dynamics, coincide (namely with the almost everywhere value of the limit above):
\begin{align*}
\dim_H \nu_\lambda & = \inf \{ \dim_H(S) : \nu_\lambda(S) = 1 \}, \\
\dim_H^- \nu_\lambda & = \inf \{ \dim_H(S) : \nu_\lambda(S) > 0 \}, \\
\dim_P \nu_\lambda & = \inf \{ \dim_P(S) : \nu_\lambda(S) = 1 \}, \\
\dim_P^- \nu_\lambda & = \inf \{ \dim_P(S) : \nu_\lambda(S) > 0 \}.
\end{align*}
Here, $\dim_P(S)$ denotes the packing dimension of the Borel set $S \subset \R$. These four dimensions are called the upper and lower Hausdorff dimension and the upper and lower packing dimension of $\nu_\lambda$, respectively; compare; for example, \cite{F97}.

\subsection{Trace Map Dynamics and Transversality}\label{ss.tracemap}

There is a fundamental connection between the spectral properties of the Fibonacci Hamiltonian and the dynamics of the \textit{trace map}
\begin{equation}\label{e.tracemapdef}
T : \Bbb{R}^3 \to \Bbb{R}^3, \; T(x,y,z)=(2xy-z,x,y).
\end{equation}
The function $G(x,y,z) = x^2+y^2+z^2-2xyz-1$ is invariant\footnote{It is usually called the Fricke-Vogt invariant.} under the action of $T$, and hence $T$ preserves the family of cubic surfaces\footnote{The surface $S_0$ is known as Cayley cubic.}
\begin{equation}\label{e.surfacelambdadef}
S_\lambda = \left\{(x,y,z)\in \Bbb{R}^3 : x^2+y^2+z^2-2xyz=1+ \frac{\lambda^2}{4} \right\}.
\end{equation}
It is therefore natural to consider the restriction $T_{\lambda}$ of the trace map $T$ to the invariant surface $S_\lambda$. That is, $T_{\lambda}:S_\lambda \to S_\lambda$, $T_{\lambda}=T|_{S_\lambda}$. We denote by $\Lambda_{\lambda}$ the set of points in $S_\lambda$ whose full orbits under $T_{\lambda}$ are bounded (it is known that $\Lambda_\lambda$ is equal to the non-wandering set of $T_\lambda$; see ).

Denote by $\ell_\lambda$ the line
\begin{equation}\label{e.loic}
\ell_\lambda = \left\{ \left(\frac{E-\lambda}{2}, \frac{E}{2}, 1 \right) : E \in \R \right\}.
\end{equation}
It is easy to check that $\ell_\lambda \subset S_\lambda$. The key to the fundamental connection between the spectral properties of the Fibonacci Hamiltonian and the dynamics of the trace map is the following result of S\"ut\H{o} \cite{S87}. An energy $E \in \R$ belongs to the spectrum $\Sigma_\lambda$ of the Fibonacci Hamiltonian if and only if the positive semiorbit of the point $(\frac{E-\lambda}{2}, \frac{E}{2}, 1)$ under iterates of the trace map $T$ is bounded. This connection shows that spectral properties of the Fibonacci Hamiltonian can be studied via an analysis of the dynamics of the trace map.

Another very important ingredient is the following. For every $\lambda > 0$, $\Lambda_\lambda$ is a locally maximal compact transitive hyperbolic set of $T_{\lambda} : S_\lambda \to S_\lambda$; see \cite{Can, Cas, DG09}. This fact allows one to use powerful tools from hyperbolic dynamics in exploring the connection between the operator and the trace map. Actually, this realization is the driving force behind all of the recent advances (roughly those dating back to 2008, starting with \cite{DEGT}). To fully exploit this, one needs that the stable manifolds of points in $\Lambda_\lambda$ intersect the line of initial conditions, $\ell_\lambda$, transversally. This crucial fact was known for $\lambda$ sufficiently large \cite{Cas} or sufficiently small \cite{DG09}, but open in the intermediate regime. As a consequence, many of the recent results could only be shown in the regimes of small and large coupling.

\begin{theorem}\label{t.transversal}
For every $\lambda > 0$, $\ell_\lambda$ intersects $W^s(\Lambda_\lambda)$ transversally.
\end{theorem}

Here, $\ell_\lambda$ denotes the line of initial conditions given in \eqref{e.loic} and $W^s(\Lambda_\lambda)$ denotes the collection of stable manifold of points in the locally maximal compact transitive hyperbolic set $\Lambda_\lambda$ of $T_{\lambda} : S_\lambda \to S_\lambda$.

Theorems~\ref{t.ddcs}, \ref{t.completegaplabeling}, and \ref{t.doesexactdim} are consequences of Theorem~\ref{t.transversal}. In fact, each of the statements in Theorems~\ref{t.ddcs}, \ref{t.completegaplabeling}, and \ref{t.doesexactdim} was previously known for $\lambda > 0$ sufficiently small \cite{DG09, DG11, DG12}; but more precisely, these statements were shown in each case to hold for all values of the coupling constant between zero and the specific value where a breakdown of transversality first occurs (or $\infty$ if no such value exists). Since transversality is easily seen to hold for $\lambda > 0$ sufficiently small \cite{DG09}, one could derive the desired statements unconditionally in the small coupling regime. For this reason, proving the absence of a breakdown of transversality had been one of the major goals in the study of the Fibonacci Hamiltonian, and Theorem~\ref{t.transversal} finally accomplishes this goal.

It is interesting to note that the proof of Theorem~\ref{t.transversal} is not a straightforward construction of an invariant cone field but rather uses the fact that the trace map is polynomial as well as spectral arguments (the fact that $\Sigma_\lambda$ does not have isolated points).

\subsection{Connections between Spectral Characteristics and Dynamical Quantities}

Recall that we are primarily interested in the following four quantities associated with the Fibonacci Hamiltonian: the upper transport exponents $\tilde \alpha^\pm_u(\lambda)$, the dimension of the spectrum $\dim_H \Sigma_\lambda$, the dimension of the density of states measure $\dim_H \nu_\lambda$, and the optimal H\"older exponent of the integrated density of states $\gamma_\lambda$. Our next main result establishes explicit identities connecting the four spectral/quantum dynamical quantities of interest with dynamical quantities associated with the trace map. In this theorem, $\mu_{\lambda,\mathrm{max}}$ denotes the measure of maximal entropy of $T_\lambda|_{\Lambda_\lambda}$ and $\mu_\lambda$ denotes the equilibrium measure of $T_\lambda|_{\Lambda_\lambda}$ that corresponds to the potential $- \dim_H \Sigma_\lambda \cdot \log \|DT_\lambda|_{E^u}\|$. Recall from \eqref{e.goldenratio} that $\varphi$ denotes the golden ratio.

\begin{theorem}\label{t.identities}
For every $\lambda > 0$, we have
\begin{align}
\tilde \alpha^\pm_u(\lambda) & = \frac{\log \varphi}{\inf_{p \in Per(T_\lambda)} \mathrm{Lyap}^u(p)}, \label{e.transportexponentidentity} \\
\dim_H \Sigma_\lambda & = \frac{h_{\mu_\lambda}}{\mathrm{Lyap}^u \mu_\lambda}, \label{e.spectrumidentity} \\
\dim_H \nu_\lambda & = \dim_H \mu_{\lambda,\mathrm{max}} = \frac{h_\mathrm{top}(T_\lambda)}{\mathrm{Lyap}^u \mu_{\lambda,\mathrm{max}}} = \frac{\log \varphi}{\mathrm{Lyap}^u \mu_{\lambda,\mathrm{max}}}, \label{e.doesmeasureidentity} \\
\gamma_\lambda & = \frac{\log \varphi}{\sup_{p \in Per(T_\lambda)} \mathrm{Lyap}^u(p)}. \label{e.hoelderexponentidentity}
\end{align}
\end{theorem}

As mentioned earlier, Theorem~\ref{t.equaltransportexponents} is a direct consequence of \eqref{e.transportexponentidentity}. Another consequence of \eqref{e.transportexponentidentity} is that we can derive explicit lower bounds for $\tilde \alpha^\pm_u(\lambda)$ by simply estimating $\inf_{p \in Per(T_\lambda)} \mathrm{Lyap}^u(p)$ from above using specific choices of periodic points. By the same token, these specific choices of periodic points will also lead to upper bounds for $\gamma_\lambda$ due to \eqref{e.hoelderexponentidentity}. For example, this leads to the following pair of explicit lower and upper bounds (the period $p$ of the periodic point leading to this bound is given in parentheses).

\begin{coro}\label{c.alphalowerbounds}
For every $\lambda > 0$, we have
\begin{align}
\gamma_\lambda& \le \frac{4 \log \varphi}{\log ( 4 \lambda^2 + \sqrt{16 \lambda^4 + 56 \lambda^2 + 45} + 7 ) - \log 2} \le \tilde \alpha^\pm_u(\lambda) \; \, \qquad \qquad (p = 6) \label{e.alphalowerboundperiod6} \\
\gamma_\lambda & \le \frac{6 \log \varphi}{\log ( \lambda^4 + \sqrt{( \lambda^4 + 8 \lambda^2 + 18 )^2 - 4} + 8 \lambda^2 + 18 ) - \log 2} \le \tilde \alpha^\pm_u(\lambda) \; (p = 4) \label{e.alphalowerboundperiod4}
\end{align}
\end{coro}

The graphs of these two functions are shown in Figure~\ref{fig:logphilyapunovexponents}. We see that for $\tilde \alpha^\pm_u(\lambda)$, \eqref{e.alphalowerboundperiod6} is better for small $\lambda$, while \eqref{e.alphalowerboundperiod4} is better for large $\lambda$, whereas the opposite is true for $\gamma_\lambda$.

\begin{figure}[t]
\includegraphics[scale=.5]{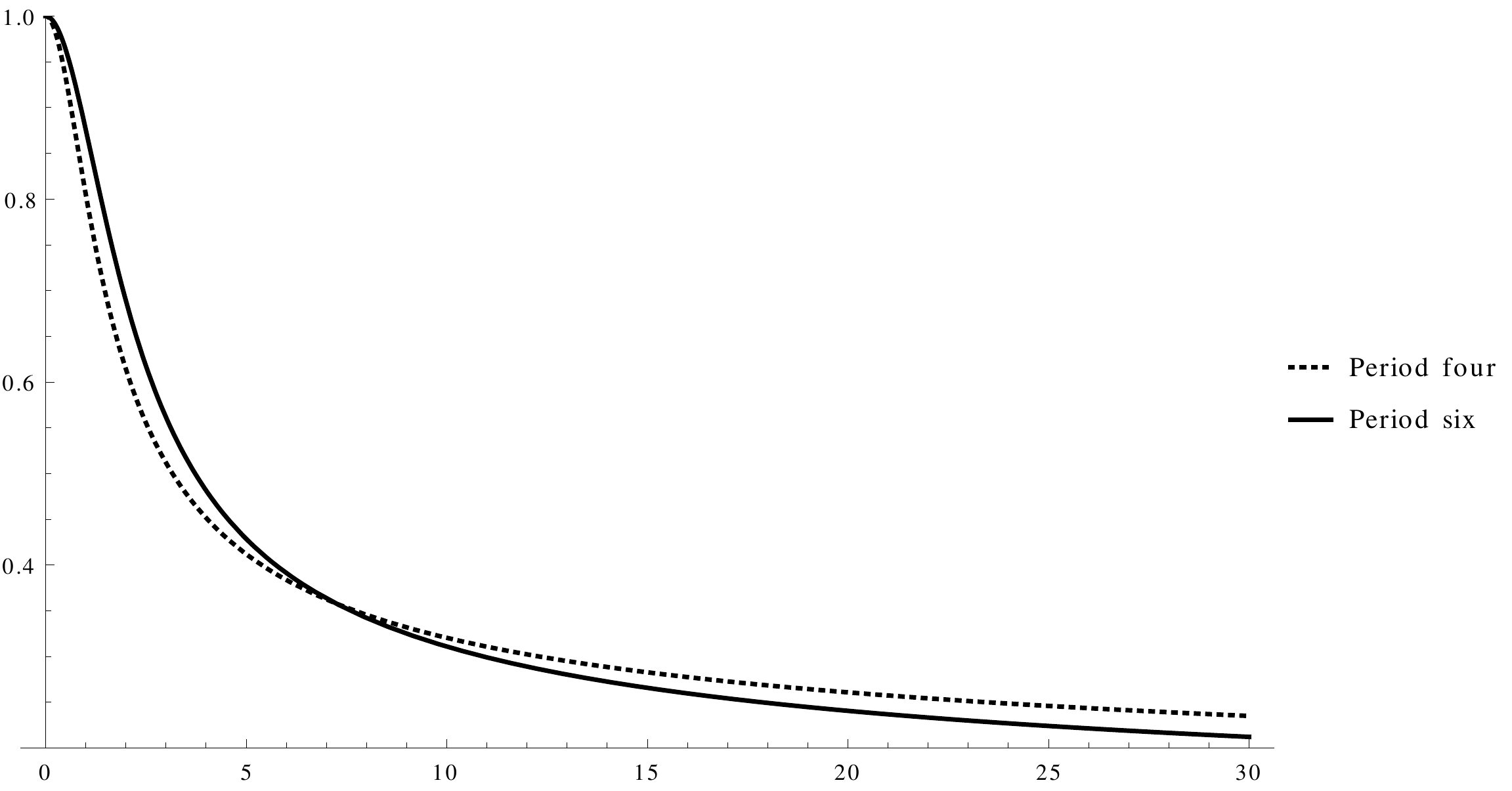}
\caption{The two bounds for $\tilde \alpha^\pm_u(\lambda)$ and $\gamma_\lambda$ from Corollary~\ref{c.alphalowerbounds}.}
\label{fig:logphilyapunovexponents}
\end{figure}

In fact, a better upper bound for $\gamma_\lambda$ can be derived via a different family of periodic points (of period two). The associated Lyapunov exponents can also be given explicitly; for the corresponding expression, see Corollary~\ref{c.holderexp}. The upper bounds resulting from the Lyapunov exponents of the families of period two (left implicit here) and period six (given above) are given in Figure~\ref{fig:logphilyapunovexponents2}.

\begin{figure}[t]
\includegraphics[scale=.5]{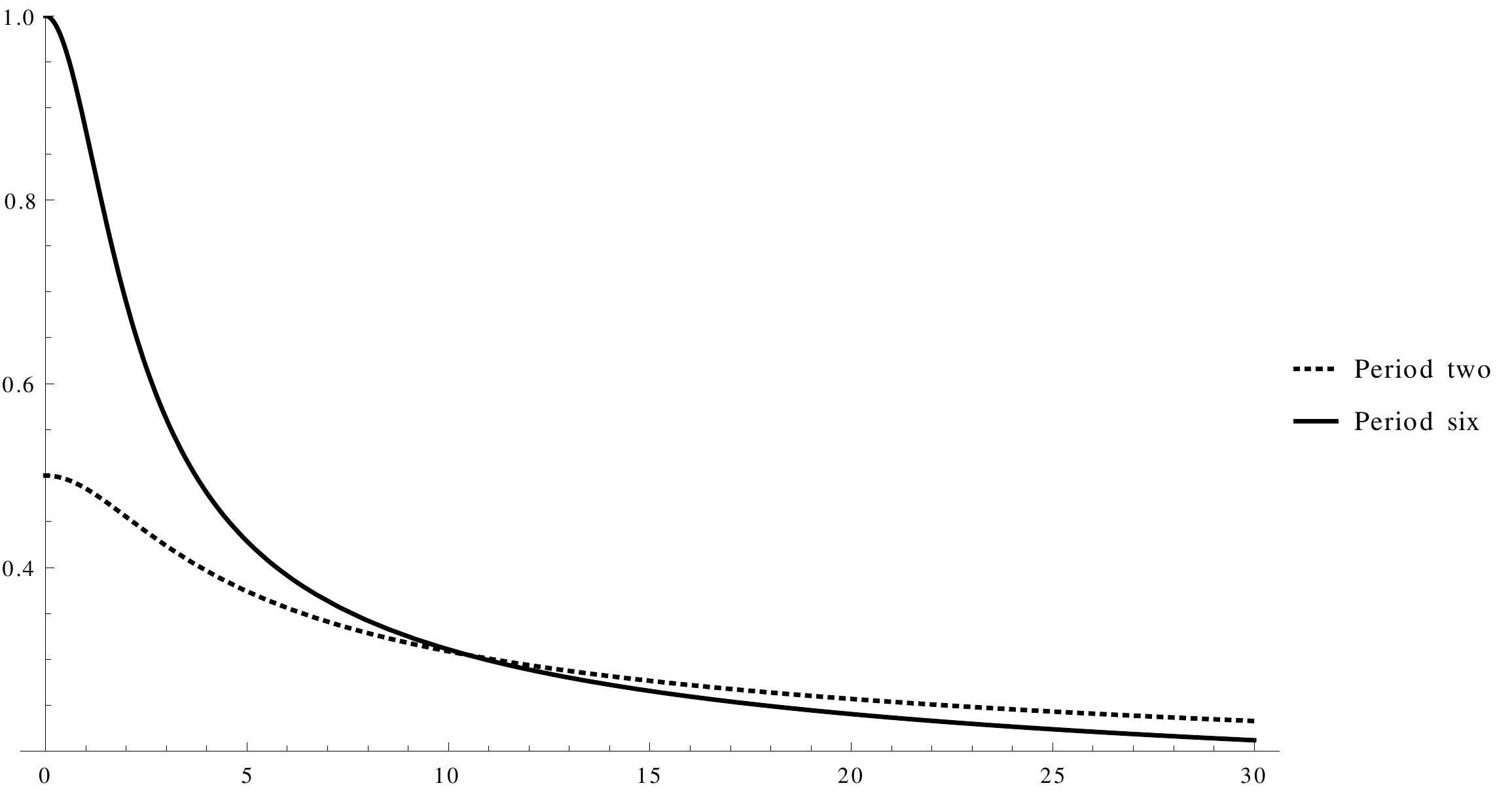}
\caption{Upper bounds for $\gamma_\lambda$ via periodic points of period $2$ and $6$.}
\label{fig:logphilyapunovexponents2}
\end{figure}

\subsection{Thermodynamical Formalism and Relations between Spectral Characteristics}

 By general principles, we have
$$
\gamma_\lambda \le \dim_H \nu_\lambda \le \dim_H \Sigma_\lambda.
$$
This is obvious since $\Sigma_\lambda$ supports the measure $\nu_\lambda$, and the almost everywhere scaling exponent of $\nu_\lambda$ is at least as big as one that works at every point. On the other hand, there is no inequality that relates $\tilde \alpha^\pm_u(\lambda)$ to one of the other three quantities, which holds for general operators.\footnote{For example, in our case at hand it turns out that $\tilde \alpha^\pm_u(\lambda)$ is strictly larger than the other three quantities, while for random potentials, $\tilde \alpha^\pm_u(\lambda)$ is strictly smaller than each of them.} The following theorem shows that for the Fibonacci Hamiltonian and every value of the coupling constant, the four quantities satisfy strict inequalities.

\begin{theorem}\label{t.strictinequalities}
For every $\lambda > 0$, we have
\begin{equation}\label{e.inequalities}
\gamma_\lambda < \dim_H \nu_\lambda < \dim_H \Sigma_\lambda < \tilde \alpha^\pm_u(\lambda).
\end{equation}
\end{theorem}

The particular inequality $\dim_H \nu_\lambda < \dim_H \Sigma_\lambda$ in \eqref{e.inequalities} establishes a conjecture of Barry Simon, which was made based on an analogy with work of Makarov and Volberg \cite{Mak, MV, V}; see \cite{DG12} for a more detailed discussion. This inequality was shown in \cite{DG12} for $\lambda > 0$ sufficiently small, and hence the conjecture had been partially established there. Our result here settles it in the generality in which it was stated.\footnote{The conjecture does not appear anywhere in print, but it was popularized by Barry Simon in many talks given by him in the past four years.}

The inequality
\begin{equation}\label{e.dimspectranspexp}
\dim_H \Sigma_\lambda < \tilde \alpha^\pm_u(\lambda)
\end{equation}
in \eqref{e.inequalities} is related to a question of Yoram Last. He asked in \cite{L96} whether in general $\dim_H \Sigma_\lambda$ bounds $\tilde \alpha^\pm_u(\lambda)$ from above and conjectured that the answer is no. The inequality \eqref{e.dimspectranspexp} confirms this. This realization is not new. It was shown in \cite{DT08} (resp., \cite{DG14}) that \eqref{e.dimspectranspexp} holds for $\lambda > 0$ sufficiently large (resp., for $\lambda > 0$ sufficiently small). What we add here is that it holds for all $\lambda > 0$.

\bigskip

The identities in Theorem~\ref{t.identities} are instrumental in our proof of Theorem~\ref{t.strictinequalities}. Indeed, once the identities \eqref{e.transportexponentidentity}--\eqref{e.hoelderexponentidentity} are established, Theorem~\ref{t.strictinequalities} can be proved using the thermodynamic formalism, which we will describe next. Define $\phi : \Lambda_\lambda \to \R$ by $\phi(x) = -\log \|DT_\lambda (x)|_{E^u}\|$ and consider the pressure function (sometimes called the Bowen function) $P : t \mapsto P(t\phi)$, where $P(\psi)$ is the topological pressure.\footnote{There are many classical books on the thermodynamical formalism; for example, \cite{B, Ru, W}. We also refer the reader to the recent introductory texts \cite{B11, Iommi, Sa}.} This function has been heavily studied; the next statement summarizes some known results; compare \cite{B, Kel, PP, Ru, W, W1}.

\begin{prop}\label{p.thermodynamic}
Suppose that $\sigma_A : \Sigma_A \to \Sigma_A$ is a topological Markov chain defined by a transitive $0-1$ matrix $A$, and $\phi : \Sigma_A \to \R$ is a H\"older continuous function. Then, the following statements hold.
\begin{itemize}

\item[{\rm (1)}] Variational principle: $P(t\phi) = \sup_{\mu \in \frak{M}} \left\{ h_\mu + t \int \phi \, d\mu \right\}$.

\item[{\rm (2)}] For every $t \in \R$, there exists a unique invariant measure $\mu_t \in \frak{M}$ {\rm (}the equilibrium state{\rm )} such that $P(t\phi) = h_{\mu_t} + t \int \phi \, d\mu_t$.

\item[{\rm (3)}] $P(t\phi)$ is a real analytic function of $t$.

\item[{\rm (4)}] If $\phi$ is cohomological to a constant, then $P(t\phi)$ is a linear function; if $\phi$ is not cohomological to a constant, then $P(t\phi)$ is strictly convex and decreasing.

\item[{\rm (5)}] For every $t_0 \in \R$, the line $h_{\mu_{t_0}} + t \int \phi \, d\mu_{t_0}$ is tangent to the graph of the function $P(t\phi)$ at the point $(t_0, P(t_0\phi))$.

\item[{\rm (6)}] Denote by $\frak{M}$ the space of $\sigma_A$-invariant Borel probability measures. The following limits exist:
$$
\lim_{t \to \infty} \int \phi \, d\mu_t = \sup_{\mu \in \frak{M}} \int \phi \, d\mu, \ \ \ \  \lim_{t \to -\infty} \int \phi \, d\mu_t = \inf_{\mu \in \frak{M}} \int \phi \, d\mu.
$$
The graph of the function $t \mapsto P(t\phi)$ lies strictly above each of the lines $t\cdot \sup_{\mu \in \frak{M}} \int \phi \, d\mu$ and $t\cdot \inf_{\mu \in \frak{M}} \int \phi \, d\mu$.

\end{itemize}
\end{prop}

Now let us return to our case where $\sigma_A : \Sigma_A \to \Sigma_A$ is conjugate to $T_\lambda|_{\Lambda_\lambda}$ and the potential is given by $\phi(x) = -\log \|DT_\lambda (x)|_{E^u}\|$ (suppressing the conjugacy). In Section~\ref{s.cohomology} we prove that this potential is not cohomological to a constant. For any $t \in \R$, consider the tangent line to the graph of $P(t)$ at the point $(t, P(t\phi))$. Since $P(t)$ is decreasing, there exists exactly one point of intersection of the tangent line with the $t$-axis, at the point $t_0 = -\frac{h_{\mu_t}}{\int \phi\, d\mu} = \frac{h_{\mu_t}}{Lyap^u\,\mu_t} = \mathrm{dim}_H\mu_t$. The last equality here is due to \cite{Man}. In particular, $\mathrm{dim}_H \mu_\mathrm{max} = \mathrm{dim}_H \nu_\lambda$ is given by the point of intersection of the tangent line to the graph of $P(t)$ at the point $(0, h_{top}(T_\lambda))$ with the $t$-axis. Also, due to Theorem~\ref{t.identities} the line $h_\mathrm{top} (T_\lambda) + t \cdot \inf_{\mu \in \frak{M}} \int \phi \, d\mu$ intersects the $t$-axis at the point $\gamma_\lambda$, and the line $h_\mathrm{top} (T_\lambda) + t \cdot \sup_{\mu \in \frak{M}} \int \phi \, d\mu$ intersects the $t$-axis at the point $\tilde \alpha^\pm_u(\lambda)$. Finally, due to \cite{MM}, the graph of $P(t)$ intersects the $t$-axis at the point $\mathrm{dim}_H\Sigma_\lambda$. These observations are illustrated in Figure~\ref{fig:thermodyn} and explain where the strict inequalities in Theorem~\ref{t.strictinequalities} come from once it is shown that $\phi$ is not cohomological to a constant (we do that in Section \ref{s.cohomology}).

\begin{figure}[t]
\includegraphics[scale=.5]{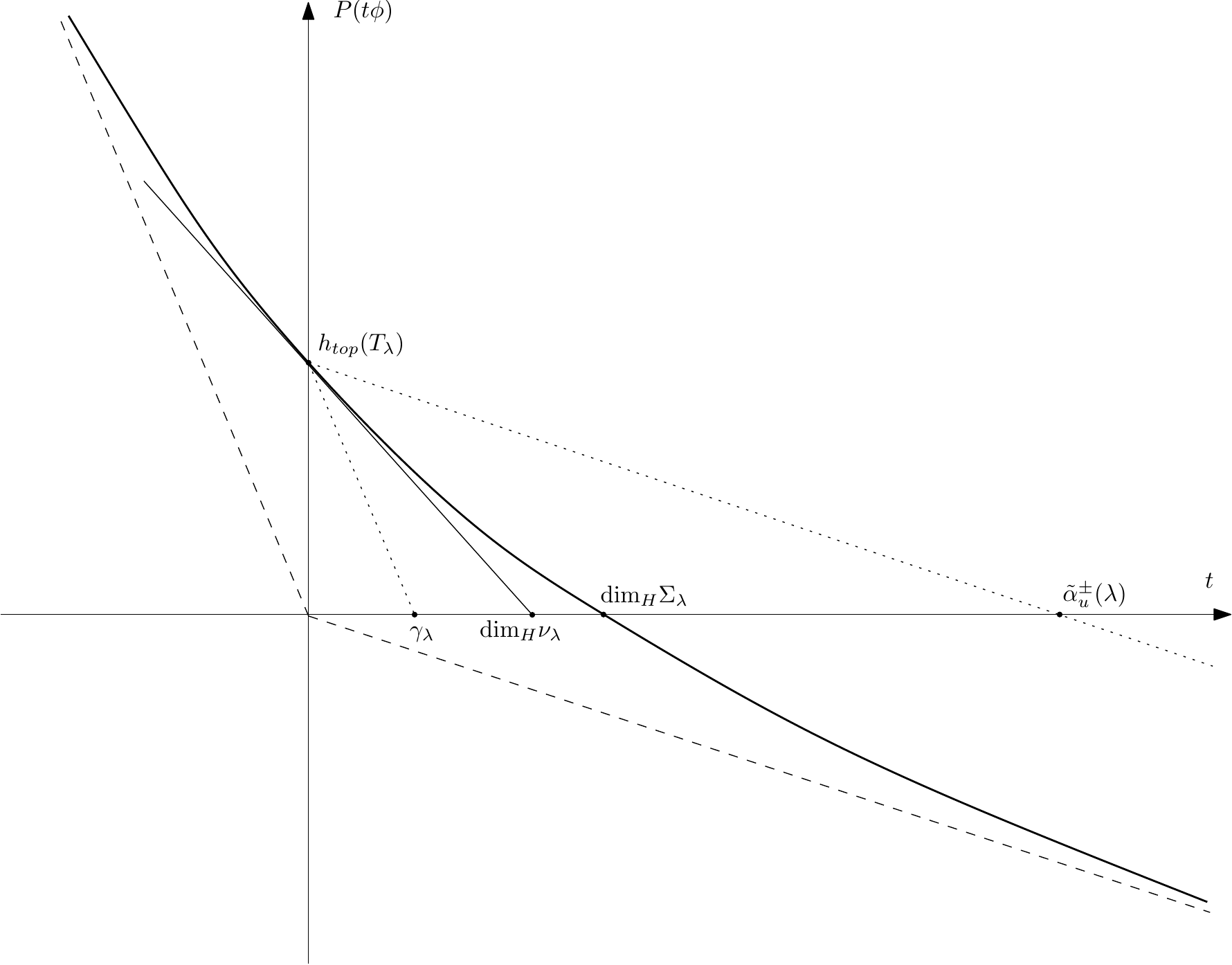}
\caption{Pressure function and spectral characteristics of the Fibonacci Hamiltonian.}
\label{fig:thermodyn}
\end{figure}

\medskip

\subsection{Large Coupling Asymptotics}

For each of the four quantities in question, the large coupling asymptotics are given in the following theorem.

\begin{theorem}\label{t.asymptotics}
We have
\begin{align}
\lim_{\lambda \to \infty} \tilde \alpha^\pm_u(\lambda) \cdot \log \lambda & = 2 \, \log \varphi, \label{e.transportexponentasymptotics} \\
\lim_{\lambda \to \infty} \dim_H \Sigma_\lambda \cdot \log \lambda & = \log (1 + \sqrt{2}) \approx 1.83156 \, \log \varphi, \label{e.spectrumasymptotics} \\
\lim_{\lambda \to \infty} \dim_H \nu_\lambda \cdot \log \lambda & = \frac{5 + \sqrt{5}}{4} \log \varphi \approx 1.80902 \, \log \varphi, \label{e.doesmeasureasymptotics} \\
\lim_{\lambda \to \infty} \gamma_\lambda \cdot \log \lambda & = 1.5 \, \log \varphi \label{e.hoelderexponentasymptotics}.
\end{align}
\end{theorem}

Only \eqref{e.doesmeasureasymptotics} is new here, the other results are stated for completeness and comparison purposes. Indeed, \eqref{e.transportexponentasymptotics} was shown in \cite{DT07, DT08}, \eqref{e.spectrumasymptotics} was shown in \cite{DEGT}, and \eqref{e.hoelderexponentasymptotics} was shown in \cite{DG13}. Thus, our proof of \eqref{e.doesmeasureasymptotics} in this paper completes our understanding of the large coupling asymptotics of the four quantities of interest.

\bigskip

Our results open the door for numerous extensions and generalizations. We briefly discuss some of them in Section~\ref{s.extensions}.

\section{Transversality}\label{sec:trans}

Recall that an invariant closed set $\Lambda$ of a diffeomorphism $f : M \to M$ is \textit{hyperbolic} if there exists a splitting of the tangent space $T_xM=E^u_x\oplus E^u_x$ at every point $x\in \Lambda$ such that this splitting is invariant under $Df$, the differential $Df$ exponentially contracts vectors from the stable subspaces $\{E^s_x\}$, and the differential of the inverse, $Df^{-1}$, exponentially contracts vectors from the unstable subspaces $\{E^u_x\}$. A hyperbolic set $\Lambda$ of a diffeomorphism $f : M \to M$ is \textit{locally maximal} if there
exists a neighborhood $U$ of $\Lambda$ such that
$$
\Lambda=\bigcap_{n\in\Bbb{Z}}f^n(U).
$$
It is known that for $\lambda > 0$, $\Lambda_{\lambda}$ is a locally maximal hyperbolic set of $T_{\lambda} : S_\lambda \to S_\lambda$; see \cite{Can, Cas, DG09}.

Recall from \eqref{e.loic} that we denote the line of initial conditions by $\ell_\lambda$. It is easy to check that $\ell_\lambda \subset S_\lambda$. An energy $E \in \Bbb{R}$ belongs to the spectrum $\Sigma_\lambda$ of the Fibonacci Hamiltonian if and only if the positive semiorbit of the point $(\frac{E-\lambda}{2}, \frac{E}{2}, 1)$ under iterates of the trace map $T$ is bounded; see \cite{S87}.

It is known that the stable manifolds of points in $\Lambda_\lambda$ intersect the line $\ell_\lambda$ transversally if $\lambda > 0$ is sufficiently small \cite{DG09} or if $\lambda \ge 16$ \cite{Cas}. It is also known, based on \cite{Bedford1993}, that if tangential intersections occur in the intermediate regime, they cannot occur at more than finitely many points. This, however, is not sufficient to state uniformly for all values of the coupling constant some of the results that are known to hold in the small and the large coupling regimes. The purpose of this section is to prove that transversality holds for all values of the coupling constant, and some of the immediate consequences; namely, we prove Theorem~\ref{t.transversal}, and its consequences -- Theorems \ref{t.ddcs}, \ref{t.completegaplabeling}, and \ref{t.doesexactdim}.

\begin{proof}[Proof of Theorem \ref{t.transversal}]
In what follows, given a curve $\eta: K\rightarrow K^n$, with $K = \R$ or $K = \C$ and $n\in\N$, $\eta^*$ denotes the image of the curve.

As we have already mentioned, transversality is known for all $\lambda > 0$ sufficiently small. Let us now assume that $\lambda_0 > 0$ is such that for all $\lambda \in (0, \lambda_0)$, transversality holds, while at $\lambda_0$, $\ell_{\lambda_0}\cap W^s(\Lambda_{\lambda_0})$ contains tangential intersections. From \cite{Bedford1993} it is known that such tangencies must be isolated; by compactness of $\ell_\lambda\cap W^s(\Lambda_\lambda)$, there is at most a finite number of such tangencies.

Let $p$ be a point of such a tangency and let $U$ be an open neighborhood of $p$ in $S_{\lambda_0}$ such that all the points of $\ell_{\lambda_0}\cap W^s(\Lambda_{\lambda_0})\cap U$ except $p$ are transversal. Notice that for each $\lambda$, $W^s(\Lambda_\lambda)$ and $\ell_\lambda$ lie on the surface $S_\lambda$. Let us first analytically project the family of laminations $W^s(\Lambda_\lambda)$ and curves $\ell_\lambda$, onto $S_{\lambda_0}$, which will include, respectively, the lamination $W^s(\Lambda_{\lambda_0})$ and the line $\ell_{\lambda_0}$.

Let us consider the complexified surfaces $S_\mathfrak{u}$, $\mathfrak{u}\in \mathcal{U}$, where $\mathcal{U}$ is a small neighborhood of $\lambda_0$ in $\C$. That is,
\begin{align*}
S_\mathfrak{u}\eqdef \set{(x,y,z)\in\C^3: x^2 + y^2 + z^2 - 2xyz - 1 = \frac{\mathfrak{u}^2}{4}}.
\end{align*}
Let us write $\hat{S}_\lambda$ for the complexification of the real surface $S_\lambda$.

By the complex-analytic version of the implicit function theorem, there exists a family of biholomorphisms $\pi(\cdot, \mathfrak{u}): S_\mathfrak{u}\rightarrow \hat{S}_{\lambda_0}$ in a neighborhood of $p$ in $\C^3$ such that $\pi(\cdot,\lambda_0)$ is the identity, $\pi$ depends holomorphically on $\mathfrak{u}$, and for all real $\mathfrak{u}$, $\pi(\cdot, \mathfrak{u})$ maps the real part of $S_\mathfrak{u}$ to $S_{\lambda_0}$.

Now let $O$ be an open neighborhood of $p$ in $\R^3$ and let $U_\lambda = S_{\lambda} \cap O$. Via $\pi(\cdot, \lambda)$, $W^s(\Lambda_\lambda)\cap U_\lambda$ is smoothly projected into $U_{\lambda_0}$. Let us denote the resulting laminations by $\mathcal{F}_\lambda$, and the lamination $W^s(\Lambda_{\lambda_0})\cap U_{\lambda_0}$ by $\mathcal{F}_{\lambda_0}$. By abuse of notation, let us denote the projection of $\ell_\lambda$ via $\pi(\cdot, \lambda)$ by the same symbol, $\ell_\lambda$.

Notice that the laminations $\mathcal{F}_\lambda$ consist of real-analytic leaves (see \cite[Section 5]{Bedford1991}), and can be included into a $C^{1 + \epsilon}$ invariant foliation. Let $\kappa$ be a parameter in the space of leaves of this foliation, such that the leaves of $\mathcal{F}_\lambda$ depend continuously on $\kappa$ in the $C^2$ topology. Moreover, each leaf of $\mathcal{F}_\lambda$ has a canonical continuation in $\lambda$ that depends holomorphically on $\lambda$ (for further details, see \cite[Section 2]{Buzzard2001}).

Let us denote by $\phi_{\lambda}^{(\kappa)}$ the leaves of $\mathcal{F}_{\lambda}$. By $\phi_{\lambda_0}^{(\kappa_0)}\in\mathcal{F}_{\lambda_0}$ we denote the leaf that admits the tangency with $\ell_{\lambda_0}$ at $p$. We will verify that the laminations $\mathcal{F}_\lambda$ satisfy the following properties.

\it
\begin{enumerate}[(i)]

\item The leaves $\phi_{\lambda}^{(\kappa)}$ as well as $\ell_{\lambda}$ admit holomorphic continuations, $\hat{\phi}_{\lambda}^{(\kappa)}$ and $\hat{\ell}_{\lambda}$, respectively, in such a way that for all $\lambda$, all intersections between $\hat{\ell}_{\lambda}$ and $\hat{\phi}_{\lambda}^{(\kappa)}$ are real.

\item For every $\lambda$ in a neighborhood of $\lambda_0$, the lamination $\mathcal{F}_\lambda$ is locally homeomorphic to a product of an interval by a Cantor set.

 \item Let $\gamma$ be a transversal to the lamination $\mathcal{F}_\lambda$. For all $\kappa_1, \kappa_2$, there exists $\Delta(\kappa_1,\kappa_2) > 0$ such that for all $\lambda$ sufficiently close to $\lambda_0$ and the leaves $\phi^{(\kappa_1)}_\lambda, \phi^{(\kappa_2)}_\lambda$ in $\mathcal{F}_\lambda$, the distance along $\gamma$ between $\gamma\cap\phi^{(\kappa_1)}_\lambda$ and $\gamma\cap\phi^{(\kappa_2)}_\lambda$ is not smaller than $\Delta$.

\end{enumerate}

\rm

\begin{proof}[Verification of (i)]
The curves $\ell_\lambda$ are complexified in a natural way (i.e. first complexify the original line of initial conditions, $\ell_\lambda$, before projection by $\pi(\cdot, \lambda)$, and then project). As for the leaves of the foliation: it is known that stable manifolds admit a suitable complexification as complex submanifolds of the complexified invariant surface $\hat{S}_\lambda$(see \cite{Bedford1991}).

To verify that all intersections should be real, we appeal to the argument given by S\"ut\H{o} in \cite{S87}: an energy $E$ belongs to the spectrum if and only if the forward orbit of the corresponding point on the line of initial conditions is bounded under the trace map. In fact, we only need the implication one way: boundedness of the forward orbit implies inclusion in the spectrum. S\"ut\H{o} considered only real values for the energy (the parameter of the line of initial conditions), since the spectrum is real. But the same argument applies verbatim to complex-valued energies.

On the other hand, a point has a bounded forward orbit if and only if it belongs to a stable manifold of $\Lambda_\lambda$ (this is known and has been used since Casdagli's work \cite{Cas}; an explicit proof was given in \cite[Corollary 2.5]{DMY13a}, see also \cite{M14}). Since the spectrum is real and the (complexified) line of initial conditions maps $\R$ into the real subspace of the invariant surface, all intersection points must be real. Now use the fact that $\pi$ preserves the real subspace.
\end{proof}

\begin{proof}[Verification of (ii)]
This holds since the nonwandering set $\Lambda_\lambda$ for the trace map restricted to $S_\lambda$, $\lambda > 0$, is a hyperbolic horseshoe (see \cite{Can, Cas, DG09}).
\end{proof}

\begin{proof}[Verification of (iii)]
This follows from a compactness argument (the lamination depends continuously on $\lambda$; restrict $\lambda$ to some compact interval around $\lambda_0$, and note that, by definition, no two distinct leaves of the lamination $\mathcal{F}_{\lambda_0}$ intersect).
\end{proof}

We will need the following simple lemma, stated here without proof, that could be derived from, for example, \cite[Theorem 1.14]{Ilyashenko2008}.

\begin{lemma}\label{lem:prelim-1}
Suppose that $\phi, \ell: \R\rightarrow\R^2$ are real-analytic, admiting complex-analytic continuations $\hat{\phi}, \hat{\ell}: \C\rightarrow\C^2$, such that $\hat{\phi}$ and $\hat{\ell}$ are injective immersions. Suppose further that at some $p\in\C$, $\hat{\phi}^*$ is tangent to $\hat{\ell}^*$ at $p$ and this tangency is isolated. Then there exists an open neighborhood $U$ of $p$ in $\C^2$ and a biholomorphism $\zeta: U\rightarrow \mathbb{D}^2=\mathbb{D}\times\mathbb{D}$ with $\mathbb{D}$ being the unit disc centered at the origin in $\C$, with the following properties.

\begin{enumerate}

\item $\zeta(p) = (0,0)$.

\item $\zeta(\Re(U)) \subseteq \Re(\mathbb{D}^2)$.

\item $\zeta$ maps the connected component of $\hat{\ell}^*\cap U$ that contains $p$ onto $\mathbb{D}$.

\item There exists a holomorphic function $f:\mathbb{D}\rightarrow\C$, such that the image of the connect component of $\hat{\phi}^*\cap U$ that contains $p$ is the graph of $f$.

\end{enumerate}

\end{lemma}

We will consider separately the case when the tangency at $p$ is quadratic, and when it is of higher order.

\begin{lemma}\label{lem:prelim-2}
If the tangency at $p$ between $\hat{\phi}^{(\kappa_0)}_{\lambda_0}$ and $\hat{\ell}_{\lambda_0}$ is of order greater than two, then there exists $\lambda\in(0, \lambda_0)$ such that $\ell_{\lambda}$ contains a point of tangency with some leaf of the lamination $\mathcal{F}_{\lambda}$.
\end{lemma}

\begin{proof}
Let us assume that the tangency at $p$ between $\hat{\phi}^{(\kappa_0)}_{\lambda_0}$ and $\hat{\ell}_{\lambda_0}$ is of order $k > 2$. Let $\zeta_0$ be a rectifying biholomorphism as in Lemma \ref{lem:prelim-1}. Then in a neighborhood of $p$, $\zeta_0$ maps $\hat{\ell}_{\lambda_0}^*$ onto $\mathbb{D}$ and $\zeta_0(\hat{\phi}_{\lambda_0}^{(\kappa_0)})$ is the graph of a holomorphic function over $\mathbb{D}$. Let us denote this function by $\hat{f}^{(\kappa_0)}_{\lambda_0}$ and its restriction onto $\R$ by $f^{(\kappa_0)}_{\lambda_0}$. Then $f_{\lambda_0}^{(\kappa_0)}$ has a root of multiplicity $k$ at the origin.

Now by holomorphic dependence on $\lambda$ of each leaf of the family $\mathcal{F}_\lambda$, as well as of $\hat{\ell}_\lambda$, we can construct a family of rectifying biholomorphisms $\set{\zeta_\lambda}$ depending holomorphically on $\lambda$ for every $\lambda$ sufficiently close to $\lambda_0$ with the same properties as $\zeta_0$ (less the tangency).

Thus $\zeta_\lambda(\mathcal{F}_\lambda)$ gives a family of laminations in a neighborhood of the origin in $\R^2$, such that for each leaf $\phi_\lambda^{(\kappa)}\in\zeta_\lambda(\mathcal{F}_\lambda)$, $\hat{\phi}_{\lambda}^{(\kappa)}$ is given as the graph of an analytic funcion $\hat{f}_\lambda^{(\kappa)}$ over $\mathbb{D}$. Notice that the resulting functions $f_\lambda^{(\kappa)}$ depend holomorphically on $\lambda$ and continuously on $\kappa$ in the $C^2$ topology. In particular, $\hat{f}^{(\kappa_0)}_\lambda\rightarrow\hat{f}^{(\kappa_0)}_{\lambda_0}$ uniformly as $\lambda\rightarrow \lambda_0$.

Hurwitz's theorem implies that for every $\lambda$ sufficiently close to $\lambda_0$, $\hat{f}_\lambda^{(\kappa_0)}$ has precisely $k > 2$ zeros (counting multiplicity) in a neighborhood of the origin, and these zeros approach the origin as $\lambda\rightarrow \lambda_0$. Since, by our standing assumptions, for $\lambda < \lambda_0$ these zeros form transverse intersections, they all must be simple. Thus when $\lambda < \lambda_0$, there are precisely $k > 2$ distinct zeros of $\hat{f}_\lambda^{(\kappa_0)}$ in a neighborhood of the origin which approach the origin as $\lambda\nearrow\lambda_0$. Furthermore, by our assumptions, these zeros are all real.

Since the biholomorphism $\zeta_\lambda$ maps $\R^2$ to $\R^2$, we obtain a family of analytic functions $f_{\lambda}^{(\kappa)}: J\rightarrow \R$ over an open interval $J\subset \R$ with $0\in J$, with the following properties.




Since $k\geq 3$, for all $\lambda < \lambda_0$ and sufficiently close to $\lambda_0$, there exist two nondegenerate compact intervals, $J_\lambda^{(1)}$ and $J_\lambda^{(2)}$ with disjoint interiors, such that the endpoints of each are given by zeros of $f_\lambda^{(\kappa_0)}$, on the interior of $J_\lambda^{(1)}$, $f_\lambda^{(\kappa_0)} < 0$ and on the interior of $J_\lambda^{(2)}$, $f_\lambda^{(\kappa_0)} > 0$, and $\abs{J_\lambda^{(i)}}\rightarrow 0$ as $\lambda\nearrow \lambda_0$, $i = 1, 2$.

By continuity, the derivative of $f_\lambda^{(\kappa_0)}$ is bounded uniformly in $\lambda$ on the interval $J$. In particular it follows that if $M_\lambda^{(i)}$ denotes the maximum of $\abs{f_\lambda^{(\kappa_0)}}$ over $J_\lambda^{(i)}$, then $M_\lambda^{(i)}\rightarrow 0$ as $\lambda\nearrow\lambda_0$.

\begin{figure}[t]
 \centering
   \subfigure[{Some $\lambda\in [\lambda', \lambda_1)$}]{
     \includegraphics[scale=.45]{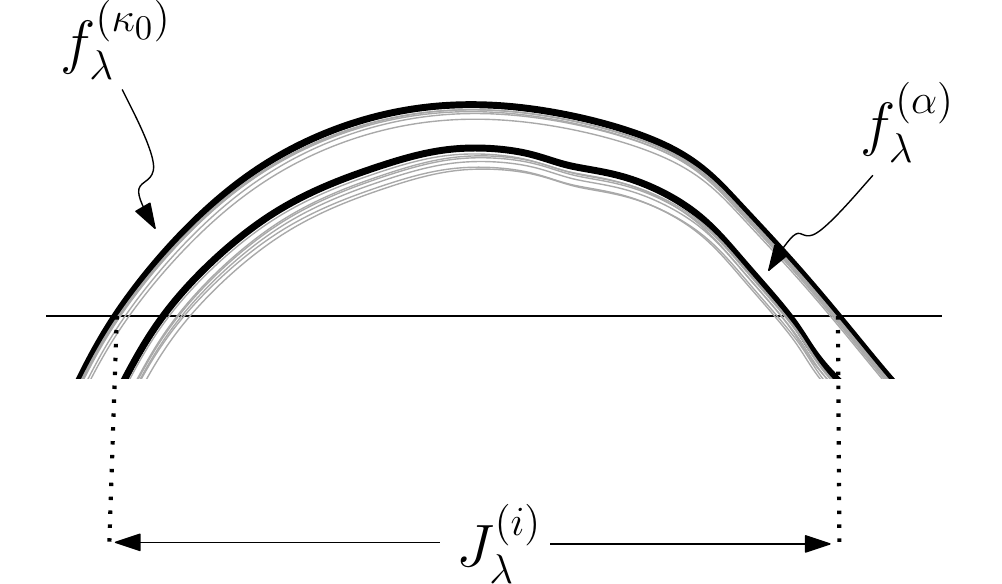}}
   \subfigure[{Some $\lambda\in (\lambda_2, \lambda_0)$}]{
     \includegraphics[scale=.45]{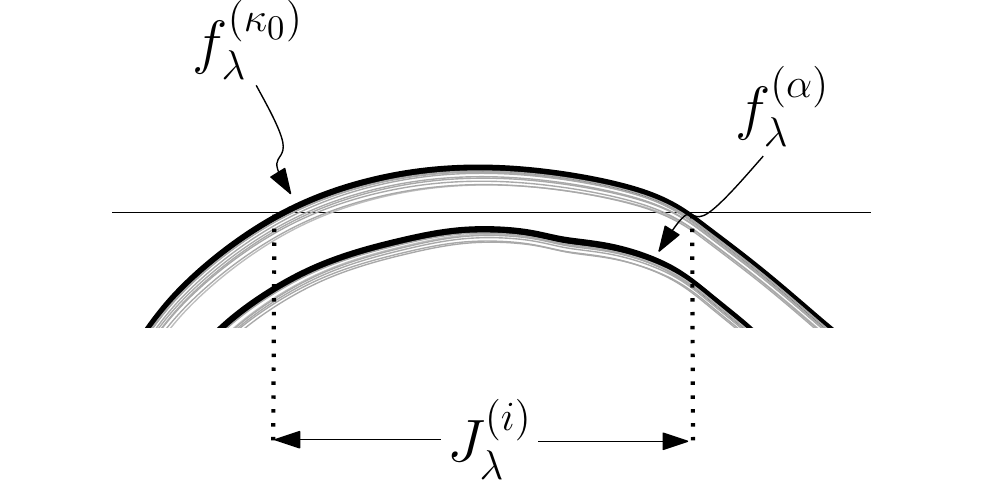}}
 \caption{}
 \label{fig:schemas1-2}
\end{figure}

\begin{figure}[t]
 \centering
   \subfigure[{$\lambda = \lambda_0$}]{
     \includegraphics[scale=.4]{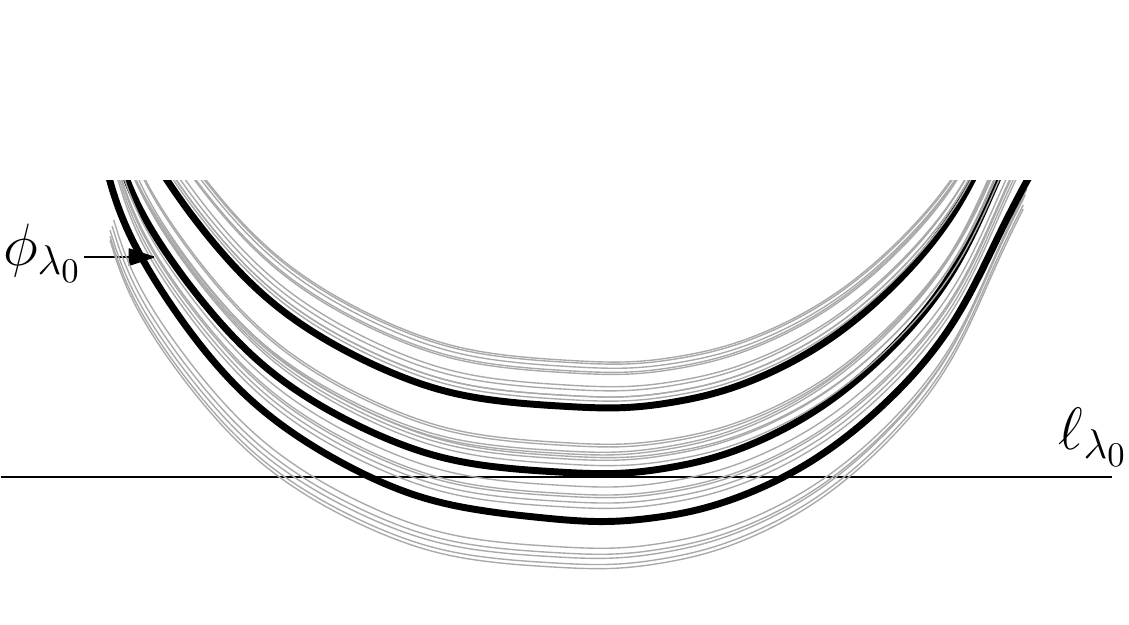}}
   \subfigure[{$\lambda < \lambda_0$}]{
     \includegraphics[scale=.4]{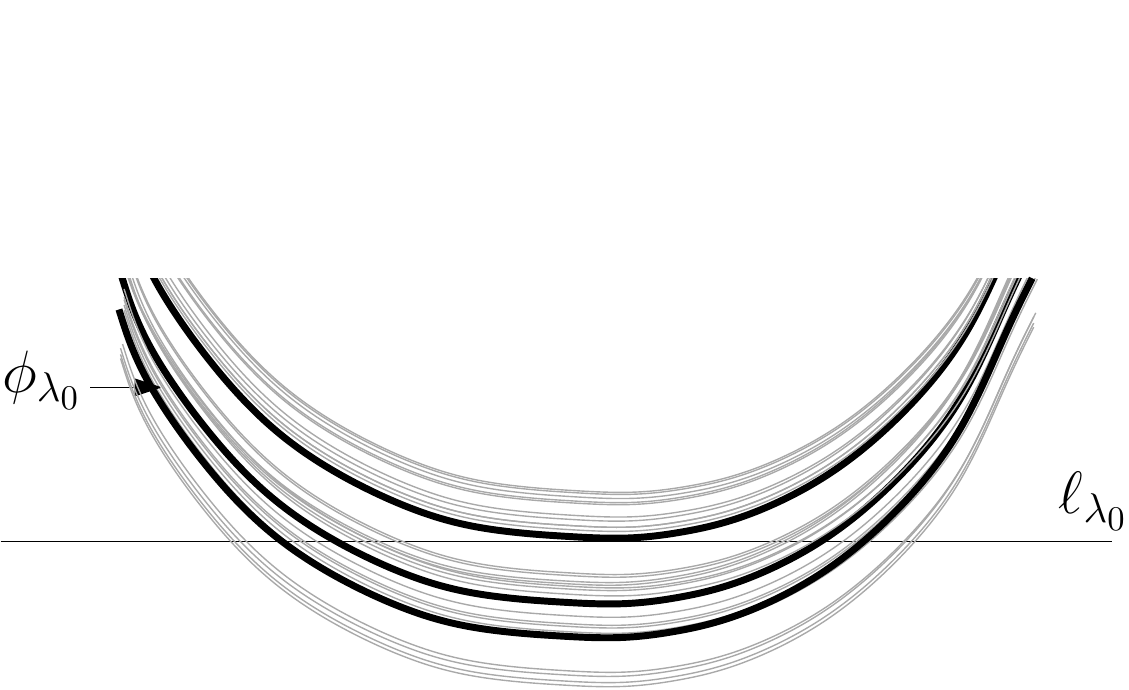}}\\
   \subfigure[{$\lambda = \lambda_0$}]{
     \includegraphics[scale=.4]{schema3.pdf}}
   \subfigure[{$\lambda < \lambda_0$}]{
     \includegraphics[scale=.4]{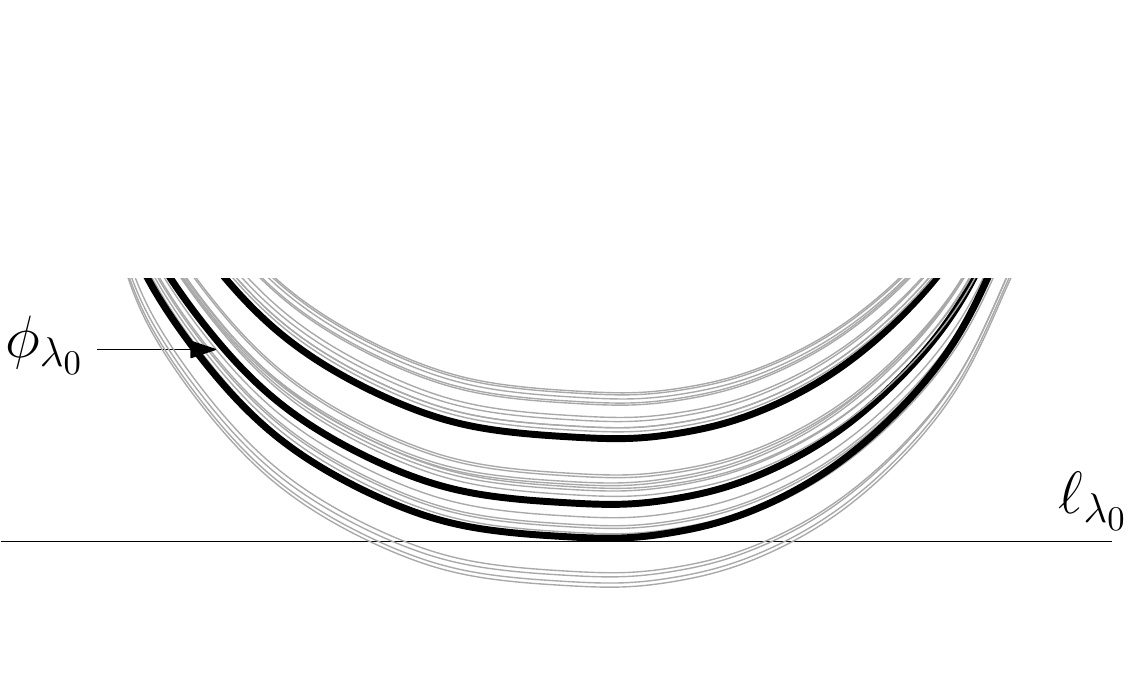}}\\
   \subfigure[{$\lambda = \lambda_0$}]{
     \includegraphics[scale=.4]{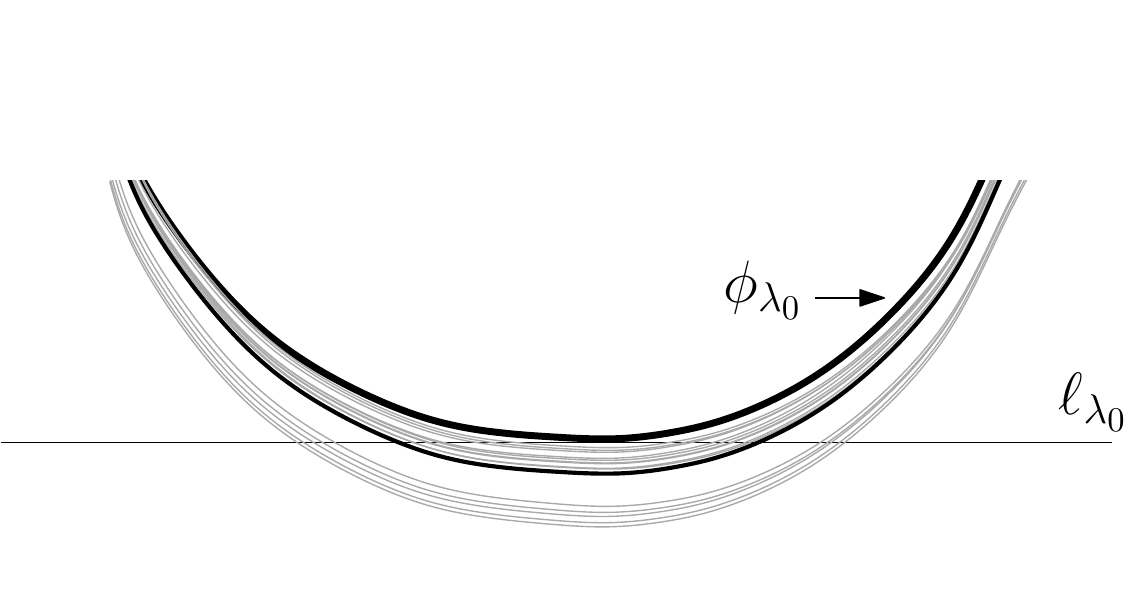}}
   \subfigure[{$\lambda < \lambda_0$}]{
     \includegraphics[scale=.4]{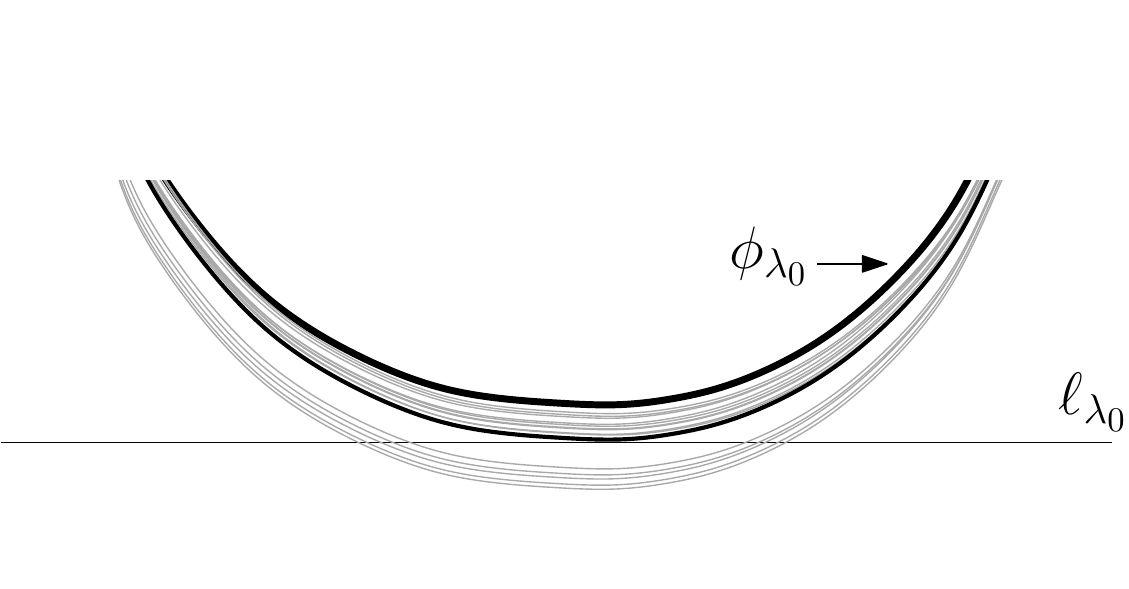}}\\
   \subfigure[{$\lambda = \lambda_0$}]{
     \includegraphics[scale=.4]{schema6.pdf}}
   \subfigure[{$\lambda < \lambda_0$}]{
     \includegraphics[scale=.4]{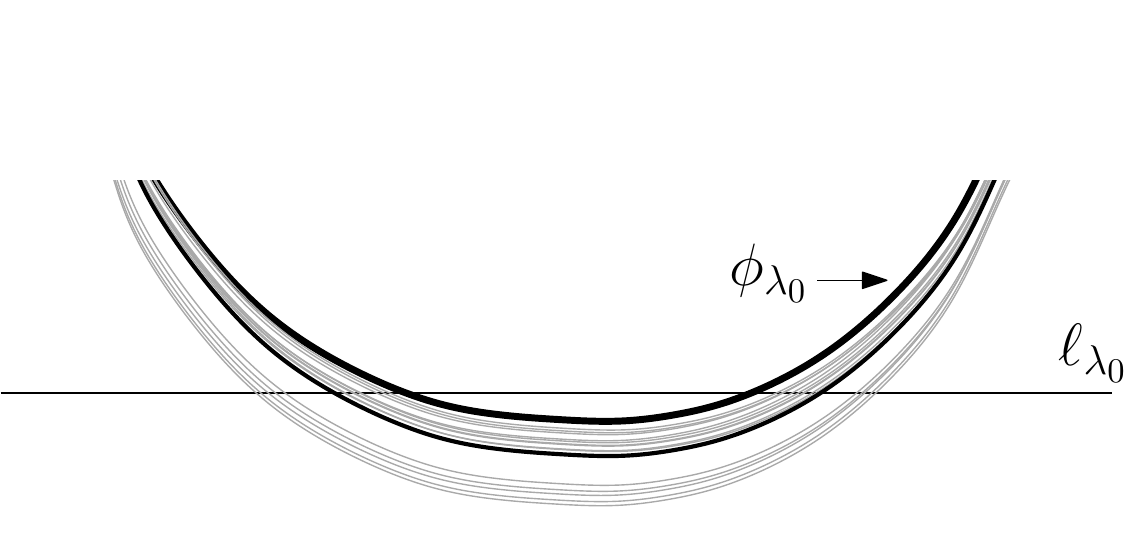}}\\
 \caption{}
 \label{fig:schemas3-8}
\end{figure}

Now fix some $\lambda' < \lambda_0$ sufficiently close to $\lambda_0$ as above. There exists $\alpha$ among the parameters $\kappa$ such that if $f_\lambda^{(\alpha)}$ is the continuation of $f_{\lambda_0}^{(\alpha)}$, then we have the following. By assumption (iii), there exists $\Delta_0 > 0$ such that for all $\lambda\in [\lambda', \lambda_0]$, we have $\abs{f_\lambda^{(\kappa_0)}-f_\lambda^{(\alpha)}} > \Delta_0$ on the interval $J_\lambda^{(i)}$. Furthermore, there exists $i\in\set{1, 2}$ such that for all $\lambda\in[\lambda', \lambda_0)$, either $f_\lambda^{(\kappa_0)}$ is negative on the interior of $J_\lambda^{(i)}$ and $f_\lambda^{(\alpha)} > f_\lambda^{(\kappa_0)}$ on $J_\lambda^{(i)}$, or $f_\lambda^{(\kappa_0)}$ is positive on the interior of $J_\lambda^{(i)}$ and $f_\lambda^{(\alpha)} < f_\lambda^{(\kappa_0)}$ on $J_\lambda^{(i)}$. Let us consider the latter case, the former being completely similar.

It follows that there exists $\lambda_1< \lambda_2 \in (\lambda', \lambda_0)$ such that for all $\lambda\in [\lambda', \lambda_1)$, the maximum of $f_\lambda^{(\alpha)}$ over $J_\lambda^{(i)}$ is positive and for all $\lambda\in(\lambda_2, \lambda_0)$, the maximum of $f_\lambda^{(\alpha)}$ over $J_\lambda^{(i)}$ is negative. As a result, there exists $\lambda\in (\lambda', \lambda_0)$ such that $f_\lambda^{(\alpha)}$ is nonpositive on $J_\lambda^{(i)}$ and has a zero $q\in J_\lambda^{(i)}$, and hence $q$ is a point of tangency (see Figure \ref{fig:schemas1-2}).
\end{proof}

We can now apply Lemma \ref{lem:prelim-2} to conclude that the tangency at $p$ must either be quadratic, or for some $\lambda < \lambda_0$, $\ell_\lambda$ intersects $\mathcal{F}_\lambda$ tangentially at some point. On the other hand, by assumption, tangencies cannot occur for $\lambda < \lambda_0$. Thus the tangency at $p$ must be quadratic.

Assume that we have a quadratic tangency at $p$ between $\ell_{\lambda_0}$ and some leaf $\phi_{\lambda_0}$ of the Cantor lamination $W^s(\Lambda_{\lambda_0})$. Assume for a moment that the leaf $\phi_{\lambda_0}$ is not a boundary of the lamination. Since for $\lambda < \lambda_0$ the tangency unfolds, it could either unfold as shown in Figure \ref{fig:schemas3-8} (a)-(b), or as in (c)-(d); in either case, arbitrarily close to $\phi_{\lambda_0}$ there exists a leaf of the foliation such that for some $\lambda < \lambda_0$, this leaf intersects the line tangentially.

Assume now that $\phi_{\lambda_0}$ is a boundary of the lamination. Since there are no isolated points in the spectrum, in this case we either have an unfolding shown in Figure \ref{fig:schemas3-8} (e)-(f) or (g)-(h).

If the tangency unfolds as shown in Figure \ref{fig:schemas3-8} (e)-(f), then as before, arbitrarily close to $\phi_{\lambda_0}$ there exists a leaf such that for some $\lambda < \lambda_0$, the intersection of the line with this leaf is tangential.

Now suppose that the tangency unfolds as shown in Figure \ref{fig:schemas3-8} (g)-(h). In this case the interval along $\ell_\lambda$ bounded by the intersection points $\phi_\lambda\cap \ell_\lambda$, as shown in the picture, corresponds to a gap in the spectrum. We know that for all sufficiently small couplings $\lambda$, all the gaps allowed by the gap labeling theorem are open (see the discussion preceding the statement of Theorem \ref{t.completegaplabeling}). On the other hand, by assumption, for all $\lambda < \lambda_0$, the line $\ell_\lambda$ intersects the stable lamination transversally; this allows for continuation of the open gaps from the small coupling regime to all $\lambda < \lambda_0$ with all gaps remaining open; see \cite[Theorem 4.3]{DG11}. In particular, this also guarantees that for any gap in the spectrum at $\lambda < \lambda_0$, its two boundary points correspond to the intersection of $\ell_\lambda$ with two stable manifolds of two distinct periodic points (for further details on gap opening, see \cite[Section 3]{DG11}). Thus an intersection of $\ell_\lambda$ for $\lambda < \lambda_0$ with a stable manifold cannot form a gap, precluding the unfolding of a tangency as shown in Figure \ref{fig:schemas3-8} (g)-(h).

This shows that the tangency at $p$ cannot be quadratic. Together with Lemma \ref{lem:prelim-2}, this proves Theorem~\ref{t.transversal}.
\end{proof}

\begin{proof}[Proof of Theorem~\ref{t.ddcs}.]
Given Theorem~\ref{t.transversal}, the result follows from \cite[Corollary~2]{DG09} and its proof.
\end{proof}

\begin{proof}[Proof of Theorem~\ref{t.completegaplabeling}.]
The result is a consequence of Theorem~\ref{t.transversal} and \cite[Theorem~4.3]{DG11}.
\end{proof}

\begin{proof}[Proof of Theorem~\ref{t.doesexactdim}.]
The assertion of the theorem can be obtained from Theorem~\ref{t.transversal} and \cite[Theorem~1.1]{DG12}; compare the discussion in Remark~(e) on \cite[p.~978]{DG12} of the role of $\lambda_0$ in the formulation of \cite[Theorem~1.1]{DG12}.
\end{proof}

\section{Transport Exponents}


In this section we prove the identity \eqref{e.transportexponentidentity}. We begin by establishing some results about the dynamics of the trace map.
\begin{prop}\label{p.trans}
For every $\lambda > 0$, all unstable manifolds of $T_\lambda : S_\lambda \to S_\lambda$ are transversal to the circle $C_\lambda := \{ z = 0 \} \cap S_\lambda$.
\end{prop}

\begin{proof}
We know that for every $\lambda > 0$ and every $k \in \Z_+$, the curve $T^k_{\lambda}(\ell_\lambda)$ has the following properties:

\

1) $T^k_{\lambda}(\ell_\lambda)$ is transversal to the plane $\{ z = c \}$ for any $c \in (-1, 1)$;

\

2) If we consider $\ell_\lambda$ as a complex line in $\mathbb{C}^3$, then $T^k_{\lambda}(\ell_\lambda)\cap \{z=0\}$ consists of $F_{k-1}$ points, and all of them are in the real subspace.

\

Indeed, both statements follow from standard results in Floquet theory. Namely, the $z$-component of $T^k_{\lambda}(\ell_\lambda(E))$ is, as a function of $E \in \C$, equal to one-half times the discriminant of a discrete Schr\"odinger operator with a periodic potential of period $F_{k-1}$ (see, e.g., \cite{S87}).\footnote{In the theory of periodic Schr\"odinger operators, the discriminant is the trace of the transfer matrix over one period.} Thus, the values of $E$ for which $T^k_{\lambda}(\ell_\lambda(E)) \in \{z=c\}$ are precisely the $E$'s for which one-half the discriminant takes on the value $c$. If $c \in (-1, 1)$, then due to, for example, \cite[Theorem~5.4.2]{SimonSzego}, there are precisely $F_{k-1}$ many of them, say $E_1, \ldots E_{F_{k-1}}$, and for every $j \in \{ 1, \ldots , F_{k-1} \}$, $E_j$ is real, and the derivative of the discriminant at $E_j$ is non-zero.

\

It is known that all unstable manifolds of $T_\lambda : S_\lambda \to S_\lambda$ are transversal to $C_\lambda$ if $\lambda$ is sufficiently small. Indeed, this is true for $\lambda = 0$ and extends to small values of $\lambda$ by continuity. Suppose that Proposition~\ref{p.trans} does not hold and denote by $\lambda^* > 0$ the smallest value of the coupling constant such that one of the unstable manifolds of $T_{\lambda^*}$ has a tangency with $C_\lambda$. Notice that this tangency cannot be quadratic. Namely, due to Theorem~\ref{t.transversal} the line $\ell_\lambda$ is transversal to the stable manifolds of $T_{\lambda^*}$, and therefore for any sufficiently large $k \in \Z_+$, the curve $T^k_{\lambda^*}(\ell_{\lambda^*})$ contains an arc that is $C^2$-close to an arc of the unstable manifold near the point of tangency. But in this case this arc would have a point of quadratic tangency with a plane $\{z = \varepsilon\}$ for some small $\varepsilon$, and this contradicts the properties of the curve $T^k_{\lambda^*}(\ell_{\lambda^*})$ above.

Therefore the tangency between $T^k_{\lambda^*}(\ell_{\lambda^*})$ and $C_{\lambda^*}$ must be of order $m > 2$. There exists a (complex) neighborhood $U \subset S_{\lambda^*}$ of the point of tangency and a biholomorphic change of coordinates $F : U \to \mathbb{D} \times \mathbb{D}$, where $\mathbb{D}$ is a unit disc in $\mathbb{C}$, such that $F(U \cap \{z=0\}) = \mathbb{D} \times \{0\}$, the point of tangency is mapped into $0$, and the arc of the unstable manifold in $U$ is mapped into the graph of a holomorphic function $g : \mathbb{D} \to \mathbb{C}$ such that $g(0) = 0$ is a zero of order $m > 2$. A holomorphic version of the Inclination Lemma (which follows, for example, from the graph transform construction from \cite[Lemma~7.5]{Ilyashenko2008}) implies that for each sufficiently large $k \in \Z_+$, there is a connected component of the intersection $T^k_{\lambda^*}(\ell_{\lambda^*})\cap U$ such that its image under $F$ is a graph of a holomorphic function $f_k : \mathbb{D} \to \mathbb{C}$ and $f_k \rightrightarrows g$. Due to the Hurwitz Theorem, for all large $k \in \Z_+$, the function $f_k$ must have $m > 2$ zeros in $\mathbb{D}$, and due to the properties of $T^k_{\lambda^*}(\ell_{\lambda^*})$, all these zeros must be simple and real. But this once again leads to the existence of a tangency between the curve $T^k_{\lambda^*}(\ell_{\lambda^*})$ and the plane $\{z = \varepsilon\}$ for some small $\varepsilon$ (due to the same arguments that were used in the proof of Lemma \ref{lem:prelim-2} above), which is a contradiction.
\end{proof}

For $E \in \C$ and $k \in \Z$, define $x_k(E)$ by
$$
T^k\left( \frac{E-\lambda}{2} , \frac{E}{2} , 1 \right) = T^k (\ell_\lambda(E)) = \left( x_{k+1}(E) , x_k(E) , x_{k-1}(E) \right) .
$$
Then, for $k \ge 0$, $x_k$ is a polynomial of degree $F_k$, where $F_0 = F_1 = 1$, $F_{k+1} = F_k + F_{k-1}$, $k \ge 1$. For $\delta > 0$, set
$$
\sigma_k^\delta = \{ E \in \C: |x_k(E)| \le 1 + \delta \}.
$$

\begin{lemma}\label{l.components}
For every $\lambda > 0$, there exists $\delta(\lambda) > 0$ such that for every $\delta \in [0,\delta(\lambda))$ and every $k \ge 0$, $\sigma_k^\delta$ has precisely $F_k$ connected components. Denote these connected components by $B_k^{(j)}(\delta)$, $j = 1 , \ldots , F_k$. Each $B_k^{(j)}(\delta)$ is symmetric about the real line, intersects $\R$ in a compact non-degenerate interval, and contains precisely one $E_k^{(j)} \in \R$ such that $x_k(E_k^{(j)}) = 0$.
\end{lemma}

\begin{remark}
{\rm (a)} We will choose a consistent labeling, namely the one which ensures that $B_k^{(j)}(\delta) \cap \R$ lies to the left of $B_k^{(j')}(\delta) \cap \R$ if $j < j'$. In particular, we have $E_k^{(1)} < E_k^{(2)} < \cdots < E_k^{(F_k)}$.\\
{\rm (b)} Clearly, the zero $E_k^{(j)}$ does not depend on $\delta \in [0,\delta(\lambda))$.
\end{remark}

\begin{proof}[Proof of Lemma~\ref{l.components}.]
Since the coefficients of the polynomial $x_k$ are real, we have $x_k(\bar E) = \overline{{x_k(E)}}$, and hence in particular $|{x_k(\bar E)}| = |x_k(E)|$. This shows that $\sigma_k^\delta$, and hence each of its connected components, is symmetric about the real line.

Recall that the free spectrum $\Sigma_0$ is equal to the interval $[-2,2]$, which corresponds to the line segment
$$
\ell_0^b = \left\{ \left(\frac{E}{2}, \frac{E}{2}, 1 \right) : E \in [-2,2] \right\} \subset \ell_0.
$$
To study the evolution of $\ell_0^b$ under the trace map, let us recall the following. The surface
$$
\mathbb{S} = S_0 \cap \{ (x,y,z)\in \Bbb{R}^3 : |x|\le 1, |y|\le 1, |z|\le 1\}
$$
is homeomorphic to $S^2$, invariant under $T$, smooth everywhere except at the four points $P_1=(1,1,1)$, $P_2=(-1,-1,1)$, $P_3=(1,-1,-1)$, and $P_4=(-1,1,-1)$, where $\mathbb{S}$ has conic singularities, and the trace map $T$ restricted to $\mathbb{S}$ is a factor of the hyperbolic automorphism of $\T^2 = \R^2 / \Z^2$ given by
\begin{equation}\label{e.torusmapdef}
\mathcal{A}(\theta_1, \theta_2) = (\theta_1 + \theta_2, \theta_1)\ (\text{\rm mod}\ 1).
\end{equation}
The semi-conjugacy is given by the map
\begin{equation}\label{e.semiconj}
F: (\theta_1, \theta_2) \mapsto (\cos 2\pi(\theta_1 + \theta_2), \cos 2\pi \theta_1, \cos 2\pi \theta_2).
\end{equation}
The map $\mathcal{A}$ is hyperbolic, and is given by the matrix $A = \begin{pmatrix} 1 & 1 \\ 1 & 0 \end{pmatrix}$.

From the explicit form \eqref{e.semiconj} of the semi-conjugacy $F$, we see that
$$
\tilde \ell_0^b = \left\{ (\theta_1 , \theta_2) : \theta_2 = 0 , \; \theta_1 \in \left[ 0,\tfrac12 \right] \right\} \subset \T^2
$$
is mapped by $F$ onto $\ell^b_0$. Since $T^k (\ell_0^b) = F(A^k (\tilde \ell_0^b))$ and $A^k (\tilde \ell_0^b)$ is the line segment from
$$
A^k \begin{pmatrix} 0 \\ 0 \end{pmatrix} = \begin{pmatrix} 0 \\ 0 \end{pmatrix}
$$
to
$$
A^k \begin{pmatrix} \frac12 \\ 0 \end{pmatrix} = \frac12 \begin{pmatrix} F_k \\ F_{k-1} \end{pmatrix}
$$
(modulo $\Z^2$), we see that $T^k (\ell_0^b)$ wraps $F_k/2$ times around $\mathbb{S}$. Now turn on $\lambda$. Since the surfaces $S_\lambda$ and the lines of initial conditions $\ell_\lambda$ change continuously, $T^k (\ell_\lambda^b)$ still wraps $F_k/2$ times around the central part of $S_\lambda$. Here, $\ell_\lambda^b$ is the line segment on $\ell_\lambda$ that corresponds to the convex hull of $\Sigma_\lambda$ via the map $E \mapsto \left( \frac{E-\lambda}{2} , \frac{E}{2} , 1 \right)$. Moreover, the extremal values reached during each turn-around (of the second coordinate, say) are now at least $1 + \frac{\lambda^2}{4}$ in absolute value. This implies that (again considering the second coordinate, say, which determines $x_k(E)$) the value of $x_k(E)$ runs at least from $-1 - \frac{\lambda^2}{4}$ to $1 + \frac{\lambda^2}{4}$ and vice versa. In particular, for every $\delta \in (0,\frac{\lambda^2}{4})$, the preimage of $[-1-\delta,1+\delta]$ under $x_k$ consists of precisely $F_k$ compact mutually disjoint intervals. This shows that $\sigma_k^\delta \cap \R$ has exactly $F_k$ connected components, each of which contains precisely one zero of $x_k$. Let us denote these $F_k$ real zeros of $x_k$ by $E_k^{(1)} < E_k^{(2)} < \cdots < E_k^{(F_k)}$.

Let us argue that each $E_k^{(j)}$ is also the only zero of $x_k$ in the complex connected component $B_k^{(j)}(\delta)$ of $\sigma_k^\delta$, which contains the real connected component that contains $E_k^{(j)}$. Suppose this fails. Since $\sigma_k^{\delta}$ is symmetric with respect to the reflection about the real axis, we can infer that if $B_k^{(j)}(\delta)$ contains another zero of $x_k$, and hence another connected component of $\sigma_k^{\delta} \cap \R$, we find that the boundary of this connected component, on which $x_k$ has constant modulus $1 + \delta$, contains a closed curve that bounds a bounded region containing points at which $x_k$ has modulus strictly larger than $1 + \delta$ (e.g., points on the real line strictly between the two connected components of $\sigma_k^{\delta} \cap \R$ in question). Thus, we obtain a contradiction due to the maximum modulus principle. It follows that $\sigma_k^\delta$, too, has precisely $F_k$ connected components, each of which contains precisely one root of $x_k$, which is real.
\end{proof}

\begin{prop}\label{p.first}
For every $\lambda > 0$ and every $\varepsilon > 0$, there exists $k_0 \in \Z_+$ such that for every $k > k_0$ and every $p \in \Lambda_\lambda$, there exists $E_k \in \R$ such that $x_k(E_k) = 0$ and
$$
\frac{1}{k}\log \|DT^k(p)|_{E^u_p}\|-\varepsilon\le \frac{1}{k}\log |x'_k(E_k)|\le \frac{1}{k}\log \|DT^k(p)|_{E^u_p}\|+\varepsilon.
$$
\end{prop}

\begin{proof}
Denote as before $C_\lambda = \{ z = 0 \} \cap S_\lambda$. Fix a small $\delta > 0$. Then there exists $k' \in \Z_+$ such that $T^{k'}(W_\delta^u(p)) \cap C_\lambda \ne \emptyset$ and $T^{-k'}(W_\delta^s(p)) \cap \ell_\lambda \ne \emptyset$ for any $p \in \Lambda_\lambda$. Choose any $p \in \Lambda_\lambda$ and pick any point $\tilde p \in T^{-k'}(W_\delta^s(p))\cap \ell_\lambda$. Let $\tau \subset \ell_\lambda$ be an interval that contains $\tilde p$ and such that $T^{k'}(\tau)$ is a connected component of $T^{k'}(\ell_\lambda)\cap U_\delta(p)$. We will denote $\tau_{k'} = T^{k'}(\tau)$ and $\tau_{k'+n} = U_\delta (T^n(p)) \cap T(\tau_{k'+n-1})$ for $n \ge 1$.

Let $k$ be sufficiently large, set $n = k-2k'$. Then $T^{k'}(\tau_{k'+n})$ must have some intersections with $C_\lambda$. Take any point $p^{**} \in C_\lambda \cap T^{k'}(\tau_{k'+n})$. Then, $p^*: = T^{-k}(p^{**}) \in \ell_\lambda$, so $p^* = \ell_\lambda(E_k)$ for some $E_k \in \R$. Let us estimate $\log |x'_k(E_k)|$. Since $W^u(\Lambda_\lambda)$ is transversal to $C_\lambda$ by Proposition~\ref{p.trans}, and $T^{k'}(\tau_{k'+n})$ is $C^{1}$-close to $T^{k'}(W_\delta^u(T^n(p)))$,
$$
\left| \log |x'_k(E_k)| - \log \|DT^k(p^*)|_{\ell_\lambda}\| \right|< C_1,
$$
where $C_1$ is some constant independent of $k$. On the other hand,
$$
\left| \log \|DT^k(p^*)|_{\ell_\lambda}\| - \log \|DT^n(T^{k'}(p^*))|_{\tau_{k'}}\| \right| < C_2,
$$
where $C_2$ is also  independent of $k$. Using \cite[Proposition~6.4.16]{KH} and the fact that $\tau_{k'+j}$ is $C^{1}$-close to $W^u_\delta(T^j(\tau_{k'}))$ (see \cite{M}) we conclude that
$$
\left| \log \|DT^n(T^{k'}(p^*))|_{\tau_{k'}}\| - \log \|DT^n(p)|_{E_p^u}\| \right| < C_3,
$$
also with a $k$-independent constant $C_3$. But this implies that for large enough $k=n+2k'$, we have
$$
\left| \frac{1}{k} \log |x'_k(E_k)| - \frac{1}{k} \log \|DT^k(p)|_{E_p^u}\| \right| < \varepsilon.
$$
\end{proof}

\begin{prop}\label{p.second}
For every $\lambda > 0$ and every $\varepsilon > 0$, there exists $k_0 \in \mathbb{N}$ such that for every $k > k_0$ and every $E_k \in \R$ with $x_k(E_k) = 0$, one can find $p \in \Lambda_\lambda$ such that
$$
\frac{1}{k} \log \|DT^k(p)|_{E^u_p}\| - \varepsilon \le \frac{1}{k} \log |x'_k(E_k)| \le \frac{1}{k} \log \|DT^k(p)|_{E^u_p}\| + \varepsilon.
$$
\end{prop}

\begin{proof}
Let us choose a ball $B_\lambda \subset \R^3$ of sufficiently large radius so that $\Lambda_\lambda \subset B_\lambda, W^s(\Lambda_\lambda) \cap \ell_\lambda \subset B_\lambda$, and $C_\lambda := \{ z = 0 \} \cap S_\lambda \subset B_\lambda$. There exist a neighborhood $U(\Lambda_\lambda)$ and $k_1 \in \mathbb{N}$ such that
\begin{enumerate}

\item if $x \in B_\lambda$ and $\mathcal{O}^+(x)\cap U(\Lambda_\lambda) = \emptyset$, then $T^n(x) \not\in B_\lambda$ for all $n > k_1$;

\item if $x \in U(\Lambda_\lambda)$ and $T(x) \not \in U(\Lambda_\lambda)$, then $\mathcal{O}^+(T(x)) \cap U(\Lambda_\lambda) = \emptyset$;

\item $U(\Lambda_\lambda)$ is inside of $\delta$-neighborhood of $\Lambda_\lambda$, where $\delta$ small enough so that Anosov Closing Lemma type arguments (more specifically, Proposition~6.4.16 from \cite{KH}) can be applied.

\end{enumerate}
Such a neighborhood $U(\Lambda_\lambda)$ can be constructed by taking a union of open rectangles around elements of a Markov partition for $\Lambda_\lambda$ so that the usual properties of Markov partitions that allow one to use coding can be applied. Slightly abusing terminology we will refer to those rectangles as the elements of a Markov partition. This will ensure that property~(2) holds. Property~(1) holds for sufficiently large $k_1$ since $\Lambda_\lambda$ is the set of bounded orbits of the map $T_\lambda$. Indeed, $\overline{B_\lambda}\backslash U(\Lambda_\lambda)$ is compact, and if (1) does not hold, one can find a sequence of points in $\overline{B_\lambda}\backslash U(\Lambda_\lambda)$ whose long finite orbits (both positive and negative) are also in that set. Any limit point would have to have a bounded orbit, but this is a contradiction since $\Lambda_\lambda$ is the set of bounded orbits of the map $T_\lambda$.

Let $k'\in \Z_+$ be such that $\bigcap_{-k' \le n \le k'} T^n(B_\lambda \cap S_\lambda) \subset U(\Lambda_\lambda)$. In this case if $x_k(E_k) = 0$ for $k \gg \max(k_1,k')$, then $T^{k'}(\ell_\lambda(E_k)) \in U_\lambda$, and also $T^n(\ell_\lambda(E_k)) \in U_\lambda$ for $n= k'+1, \ldots, k-k_1$. By the choice of $B_\lambda$, we have $C_\lambda\subset B_\lambda$, and since $x_k(E_k)=0$, we also have $T^k(\ell_\lambda(E_k))\in C_\lambda\subset B_\lambda$. Set $P = T^{k'}(\ell_\lambda(E_k))$, $T^i(P) \in U(\Lambda_\lambda)$ for all $0 \le i \le k-k_1-k'$. Let $\bar{p} \in \Lambda_\lambda$ be any point that has the same symbolic dynamics over the finite time interval of length $k-k_1-k'$. In other words, $\bar{p}$ is such that $T^i(\bar{p})$ and $T^i(P)$ belong to the same element of the Markov partition of $\Lambda_\lambda$ for $i=0, 1, \ldots, k-k'-k_1$. In this case $\mathrm{dist} (T^i(\bar{p}), T^i(P)) \le \delta$ for $i=0, 1, \ldots, k-k'-k_1$. This implies (see Proposition~6.4.16 from \cite{KH}) that in fact $\mathrm{dist} (T^j(\bar{p}), T^j(P)) \le C \rho^{\min(j, m-j)}\delta$ for some $\rho < 1$, where $m = k-k'-k_1$ and $0 \le j \le m$. Distortion estimates imply now that
$$
\left|\log \|DT^m(\bar{p})|_{E^u_{\bar{p}}}\|- \log \|DT^m(P)|_{T^{k'+k_1}(\ell_\lambda)}\|\right|\le C,
$$
where the constant $C$ is independent of $m$. Take $p = T^{-k'}(\bar{p})$. Then, for some $C'$ independent of $m$, we have
\begin{multline*}
 \left|\log \|DT^k({p})|_{E^u_{p}}\|-\log \|DT^m(\bar{p})|_{E^u_{\bar{p}}}\|\right|+\\
 +\left|\log \|DT^k(\ell_\lambda(E_k))|_{\ell_\lambda}\|-\log \|DT^m(P)|_{T^{k'+k_1}(\ell_\lambda)}\|\right|\le C'
\end{multline*}
and hence
$$
\left| \frac{1}{k} \log \|DT^k({p})|_{E^u_{p}}\| - \frac{1}{k} \log \|DT^k(\ell_\lambda(E_k))|_{\ell_\lambda}\| \right| \le \frac{(C+C')}{k} \le \varepsilon
$$
if $k$ is sufficiently large. Together with the fact that
$$
\left| \log |x'_k(E_k)| - \log \|DT^k(\ell_\lambda(E_k))|_{\ell_\lambda}\| \right| < C_1
$$
with $C_1$ independent of $k$,  this proves Proposition~\ref{p.second}.
\end{proof}

\begin{lemma}\label{l.clearin}
We have
$$
\lim_{k \to \infty} \frac{1}{k} \inf_{p \in \Lambda_\lambda} \log\|DT^k(p)|_{E^u_p}\| = \inf_{p \in \Lambda_\lambda} \mathrm{Lyap}^u(p) = \inf_{p \in Per(\Lambda_\lambda)} \mathrm{Lyap}^u(p).
$$
That is, the limit on the left-hand side exists and equals the other two expressions.
\end{lemma}

\begin{proof}
Notice that we certainly have
$$
\liminf_{k \to \infty} \frac{1}{k} \inf_{p \in \Lambda_\lambda} \log \|DT^k(p)|_{E^u_p}\| \le \inf_{p \in \Lambda_\lambda} \mathrm{Lyap}^u(p) \le \inf_{p \in Per(\Lambda_\lambda)} \mathrm{Lyap}^u(p).
$$
Let us show that
\begin{equation}\label{e.A}
A := \liminf_{k \to \infty} \frac{1}{k} \inf_{p \in \Lambda_\lambda} \log \|DT^k(p)|_{E^u_p}\| \ge \inf_{p \in Per(\Lambda_\lambda)} \mathrm{Lyap}^u(p).
\end{equation}
Fix an arbitrarily small $\varepsilon>0$.  There exist $k_j \to \infty$ and $p_j \in \Lambda_\lambda$ such that
$$
\frac{1}{k_j} \log \|DT^{k_j}(p_j)|_{E^u_{p_j}}\| \le A + \varepsilon.
$$
The specification property (see, for example, \cite[Theorem~18.3.9]{KH}) implies that for any $\delta > 0$, we can find a sequence of periodic orbits $\{q_j\}$ such that

\

1) $T^{k_j+M}(q_j)=q_j$, where $M\in \mathbb{N}$ is independent of $j\in \mathbb{N}$;

\

2) $\mathrm{dist} (T^i(q_j), T^i(p_j)) \le \delta$ for $i=0, \ldots, k_j-1$.

\

Now the quantitative version of the Anosov Closing Lemma (see, e.g., \cite[Proposition 6.4.16]{KH}) implies that in fact for some $\rho < 1$,
$$
\mathrm{dist} (T^i(q_j), T^i(p_j)) \le C \rho^{\min(i, k_j-i)} \delta.
$$

 The stable and unstable distributions of a two dimensional horseshoe are $C^1$ , see \cite[Corollary~19.1.11]{KH}. Now smoothness of the unstable bundle $\{E^u_x\}_{x\in \Lambda_\lambda}$
 allows us to use standard distortion estimates and hence to deduce that
$$
\log \|DT^{k_j}(q_j)|_{E^u_{q_j}}\| \le \log \|DT^{k_j}(p_j)|_{E^u_{p_j}}\| + C',
$$
where the constant $C'$ is independent of $j$. Hence for large enough $k_j$, we have
$$
\mathrm{Lyap}^u(q_j) = \frac{1}{k_j+M} \log \|DT^{k_j+M}(q_j)|_{E^u_{q_j}}\| \le \frac{1}{k_j} \log \|DT^k_j(p_j)|_{E^u_{p_j}}\| + \varepsilon + \frac{C'}{k_j} \le A + 3 \varepsilon.
$$
This implies that
$$
\inf_{p \in Per(\Lambda_\lambda)} \mathrm{Lyap}^u(p) \le A + 3 \varepsilon,
$$
and since $\varepsilon > 0$ can be chosen arbitrary small, we have
$$
\inf_{p \in Per(\Lambda_\lambda)} \mathrm{Lyap}^u(p) \le A.
$$
This completes the proof of the inequality \eqref{e.A}.

Now we need to show that
\begin{equation}\label{e.B}
B := \limsup_{k \to \infty} \frac{1}{k} \inf_{p \in \Lambda_\lambda} \log \|DT^k(p)|_{E^u_p}\| \le \inf_{p \in Per(\Lambda_\lambda)} \mathrm{Lyap}^u(p).
\end{equation}
Once again, fix an arbitrarily small $\varepsilon > 0$. Take a periodic point $p_0 \in Per(\Lambda_\lambda)$, $T^m(p_0) = p_0$, such that
$$
\mathrm{Lyap}^u(p_0) \le \inf_{p \in Per(\Lambda_\lambda)} \mathrm{Lyap}^u(p) + \varepsilon.
$$
For all sufficiently large $k$, we have
$$
\frac{1}{k} \log \|DT^k(p_0)|_{E^u_{p_0}}\| \le \mathrm{Lyap}^u(p_0) + \varepsilon \le \inf_{p \in Per(\Lambda_\lambda)} \mathrm{Lyap}^u(p) + 2 \varepsilon,
$$
hence
$$
\frac{1}{k} \inf_{p \in \Lambda_\lambda} \log \|DT^k(p)|_{E^u_p}\| \le \frac{1}{k} \log \|DT^k(p_0)|_{E^u_{p_0}}\| \le \mathrm{Lyap}^u(p_0) + \varepsilon \le \inf_{p \in Per(\Lambda_\lambda)} \mathrm{Lyap}^u(p) + 2 \varepsilon.
$$
Therefore
$$
B = \limsup_{k \to \infty} \frac{1}{k} \inf_{p \in \Lambda_\lambda} \log\|DT^k(p)|_{E^u_p}\| \le \inf_{p \in Per(\Lambda_\lambda)} \mathrm{Lyap}^u(p) + 2 \varepsilon,
$$
and since $\varepsilon > 0$ is arbitrary, we have $B \le \inf_{p \in Per(\Lambda_\lambda)} \mathrm{Lyap}^u(p)$. Together with \eqref{e.A} this completes the proof of Lemma~\ref{l.clearin}.
\end{proof}

As a direct corollary of Propositions~\ref{p.first} and \ref{p.second} and Lemma \ref{l.clearin} we get the following statement:

\begin{prop}\label{p.xkprimelyap}
We have
$$
\lim_{k \to \infty} \frac{1}{k} \log \min_{j=1, \ldots, F_k} \left| x'_k(E_k^{(j)}) \right| = \inf_{p \in Per(\Lambda_\lambda)} \mathrm{Lyap}^u(p).
$$
That is, we have that the limit on the left-hand side exists and that it is equal to the right-hand side.
\end{prop}


Recall that we considered above the sets $\sigma_k^\delta$ and their connected components $B_k^{(j)}(\delta)$. Define further
\begin{align*}
r_k^{(j)}(\delta) & = \sup \{ r > 0 : B(E_k^{(j)} , r) \subseteq B_k^{(j)}(\delta) \}, \quad \; \;
r_k(\delta) = \max_{j = 1,\ldots,F_k} r_k^{(j)}(\delta), \\
R_k^{(j)}(\delta) & = \inf \{ R > 0 : B(E_k^{(j)} , R) \supseteq B_k^{(j)}(\delta) \}, \quad
R_k(\delta) = \max_{j = 1,\ldots,F_k} R_k^{(j)}(\delta).
\end{align*}
The identity \eqref{e.transportexponentidentity} will follow from Proposition~\ref{p.xkprimelyap} and the following proposition.

\begin{prop}\label{p.transportbounds}
{\rm (a)} For every $\lambda > 0$ and $\delta \in (0,\delta(\lambda))$, we have
\begin{equation}\label{e.alphaulowerbound}
\tilde \alpha_u^- \ge \frac{\log \varphi}{\limsup_{k \to \infty} \frac{1}{k} \log \frac{1}{r_k(\delta)}}
\end{equation}
and
\begin{equation}\label{e.alphauupperbound}
\tilde \alpha_u^+ \le \frac{\log \varphi}{\liminf_{k \to \infty} \frac{1}{k} \log \frac{1}{R_k(\delta)}}.
\end{equation}

{\rm (b)} For every $\lambda > 0$ and $\delta \in (0,\delta(\lambda)/2)$, we have
\begin{equation}\label{e.rkRkxprimeconnection1}
\frac{1}{R_k(\delta)} \ge \left(\frac{\delta}{(1 + \delta)(1 + 2\delta)} \right)^2 \left( \min_j |x_k'(E_k^{(j)})| \right)
\end{equation}
and
\begin{equation}\label{e.rkRkxprimeconnection2}
\frac{1}{r_k(\delta)} \le \frac{(2 + 3\delta)^2}{(1 + \delta)(1 + 2 \delta)^2} \left( \min_j |x_k'(E_k^{(j)})| \right)
\end{equation}
for every $k \ge 0$.

{\rm (c)} For $\lambda > 0$ and $\delta \in (0,\delta(\lambda)/2)$, we have
$$
\tilde \alpha_u^- \ge \frac{\log \varphi}{\limsup_{k \to \infty} \frac{1}{k} \log \left( \min_{j = 1,\ldots,F_k} \left| x_k'(E_k^{(j)}) \right| \right)}.
$$
and
$$
\tilde \alpha_u^+ \le \frac{\log \varphi}{\liminf_{k \to \infty} \frac{1}{k} \log \left( \min_{j = 1,\ldots,F_k} \left| x_k'(E_k^{(j)}) \right| \right)}.
$$
\end{prop}

\begin{proof}
(a) The strategy of proving \eqref{e.alphaulowerbound} and \eqref{e.alphauupperbound} is inspired by \cite{DG14, DT07, DT08}. The Parseval identity implies (see, e.g., \cite[Lemma~3.2]{kkl})
\begin{equation}\label{e.parsform}
2\pi \int_0^{\infty} e^{-2t/T} | \langle \delta_n , e^{-itH} \delta_0 \rangle |^2 \, dt = \int_{-\infty}^\infty \left|\langle \delta_n  , (H - E - \tfrac{i}{T})^{-1} \delta_0 \rangle \right|^2 \, dE,
\end{equation}
and hence for the time averaged outside probabilities, defined by
\begin{equation}\label{e.taop}
\langle P(N,\cdot) \rangle (T) = \frac{2}{T} \int_0^{\infty} e^{-2t/T} \sum_{|n| \ge N} | \langle \delta_n , e^{-itH} \delta_0 \rangle |^2 \, dt,
\end{equation}
we have
\begin{equation}\label{e.taopresform}
\langle P(N,\cdot) \rangle (T) = \frac{1}{\pi T} \sum_{|n| \ge N} \int_{-\infty}^\infty \left|\langle \delta_n  , (H - E - \tfrac{i}{T})^{-1} \delta_0 \rangle \right|^2 \, dE.
\end{equation}
The right-hand side of \eqref{e.taopresform} may be studied by means of transfer matrices at complex energies, which are defined as follows. For $z \in \C$, $n \in \Z$, we set
$$
M(n;\omega,z) = \begin{cases} T(n;\omega,z) \cdots T(1;\omega,z) & n \ge 1, \\ T(n;\omega,z)^{-1} \cdots T(-1;\omega,z)^{-1} & n \le -1, \end{cases}
$$
where
$$
T(\ell;\omega,z) = \begin{pmatrix} z - \lambda \chi_{[1-\alpha,1)}(\ell \alpha + \omega \!\!\!\! \mod 1) & -1 \\ 1 & 0 \end{pmatrix}.
$$
The following statement follows from \cite[Proposition~2]{DT08}: For every $\lambda , \delta > 0$, there are constants $C,\xi$ such that for every $k$, every $z \in \sigma_k^\delta$, and every $\omega \in \T$, we have
\begin{equation}\label{e.2sidedpowerlaw}
\| M(n;\omega,z) \| \le C n^\xi.
\end{equation}
for $1 \le |n| \le F_k$. Combining ideas from the proof of \cite[Proposition~2]{DT08} and the proof of \cite[Theorem~5.1]{DG11}, one can show the following for the exponent $\xi$ in \eqref{e.2sidedpowerlaw}. If we denote the largest root of the polynomial $x^3 - (2+\lambda) x - 1$ by $a_\lambda$ (note that for small $\lambda > 0$, we have $a_\lambda \approx \varphi + c\lambda$ with a suitable constant $c$), then for any
\begin{equation}\label{e.gammaest}
\xi > 2 \frac{\log [(5 + 2\lambda)^{1/2} (3 + \lambda) a_\lambda]}{\log \varphi},
\end{equation}
there is a constant $C$ such that \eqref{e.2sidedpowerlaw} holds for $z \in \sigma_k^\delta$ and $\omega \in \T$.

Let us now consider $\lambda > 0$, $\delta \in (0,\delta(\lambda))$, and $\varepsilon > 0$. Consider the value of $j \in \{1, \ldots , F_k\}$ with $r_k^{(j)}(\delta) = r_k(\delta)$. By definition, $E_k^{(j)}$ is the only zero of $x_k$ in $B_k^{(j)}(\delta)$.

For $\rho > 0$ arbitrary, consider
\begin{equation}\label{e.sdef}
s = \frac{\limsup_{k \to \infty} \frac1k \log \frac{1}{r_k(\delta)}}{\log \varphi} + \rho.
\end{equation}
Clearly, $s$ is strictly positive. By definition of $s$, for suitably chosen $C_\delta > 0$, we have
\begin{equation}\label{e.sprop}
C_\delta F_k^{s} \ge \frac{2}{r_k(\delta)}
\end{equation}
for every $k \ge 0$.

Take $N = F_k$ and consider $T \ge C_\delta N^{s}$ (which in turn implies $T \ge \frac{2}{r_k(\delta)}$ by \eqref{e.sprop}). Due to the Parseval formula \eqref{e.parsform}, we can bound the time-averaged outside probabilities from below as follows,
\begin{equation}\label{parseval}
\langle P(N,\cdot) \rangle (T) \gtrsim \frac1T \int_\R \left( \max \left\{ \|M(N;\omega,E+i/T)\|, \|M(-N;\omega,E+i/T)\| \right\} \right)^{-2} \, dE.
\end{equation}
See, for example, the proof of \cite[Theorem~1]{DT03} for an explicit derivation of \eqref{parseval} from \eqref{e.parsform}.

To bound the integral from below, we integrate only over those $E \in (E_k^{(j)}-r_k(\delta), E_k^{(j)}+r_k(\delta))$ for which $E+i/T \in B(E_k^{(j)}, r_k(\delta)) \subset B_k^{(j)}(\delta)$. Since $\frac{1}{T} \le \frac{r_k(\delta)}{2}$, the length of such an interval $I_k$ is larger than $cr_k(\delta)$ for some suitable $c > 0$. For $E \in I_k$, we have
$$
\|M(N;\omega, E+i\varepsilon)\| \lesssim N^{\xi} \lesssim T^{\frac{\xi}{s}}.
$$
Therefore, \eqref{parseval} together with \eqref{e.2sidedpowerlaw} gives
\begin{equation}\label{eq1}
\langle P(N,\cdot) \rangle (T) \gtrsim \frac{r_k}{T} \, T^{-\frac{2\xi}{s}} \gtrsim T^{-2-\frac{2\xi}{2}},
\end{equation}
where $N = F_k$, $T \ge C_\delta N^s$, for any $k \ge k_0$.

Now let us take any sufficiently large $T$ and choose $k$ maximal with $C_\delta F_k^s \le T$. Then,
$$
C_\delta F_k^s \le T < C_\delta F_{k+1}^s \le 2^s C_\delta F_k^s.
$$
It follows from \eqref{eq1} that
$$
\left\langle P \left( \tfrac{1}{2 C_\delta^{1/s}} T^{\frac{1}{s}},\cdot \right) \right\rangle (T) \ge \langle P(F_k,\cdot) \rangle (T) \gtrsim T^{-2-\frac{2\xi}{s}}
$$
for all sufficiently large $T$. It follows from the definition of $\tilde \beta^-(p)$ and $\tilde \alpha_u^-$ that
$$
\tilde \beta_{\delta_0}^-(p) \ge \frac{1}{s} - \frac{2}{p}\left( 1 + \frac{\xi}{s} \right)
$$
and
$$
\tilde \alpha_u^- \ge \frac{1}{s} = \left( \frac{\limsup_{k \to \infty} \frac1k \log \frac{1}{r_k(\delta)}}{\log \varphi} + \rho \right)^{-1},
$$
by \eqref{e.sdef}. Since $\rho > 0$ can be taken arbitrarily small, this proves \eqref{e.alphaulowerbound}.

\bigskip

Let us recall \cite[Lemma~4]{DT07}: Given any $\delta > 0$ and $E \in \C$, a necessary and sufficient condition for $\{x_k(E)\}_{k \ge -1}$ to be unbounded is that
\begin{equation}\label{critN}
|x_{K-1}(E)| \le 1 + \delta, \quad |x_{K}(E)| > 1 + \delta, \quad |x_{K+1}(E)| > 1 +
\delta
\end{equation}
for some $K \ge 0$. This $K$ is unique. Moreover, in this case we have
\begin{equation}\label{expgrowth}
|x_{K+k}(E)| \ge (1 + \delta)^{F_k} \text{ for } k \ge 0.
\end{equation}

By definition of $R_k(\delta)$, we have
$$
\sigma_k^\delta \subseteq \{ z \in \C : |\mathrm{Im} \ z| \le R_k(\delta) \}.
$$
We set
\begin{equation}\label{e.sprimedef}
s' = \frac{\liminf_{k \to \infty} \frac1k \log \frac{1}{R_k(\delta)}}{\log \varphi} - \rho'
\end{equation}
for $\rho' > 0$ small enough so that $s' > 0$ (Proposition~\ref{p.xkprimelyap} shows that it is possible to find such a $\rho'$ since the right-hand side in that proposition is positive as $\Lambda_\lambda$ is a hyperbolic set), and then choose some suitable $C_\delta' > 0$, so that we have
$$
R_k(\delta) < C_\delta' F_k^{- s'},
$$
for every $k \ge 0$. In particular,
\begin{equation}\label{imwidth}
\sigma_k^\delta \cup \sigma_{k+1}^\delta \subseteq \{ z \in \C : |\mathrm{Im} \ z| < C_\delta' F_k^{- s'} \}.
\end{equation}

For each $\varepsilon = \mathrm{Im} \ z > 0$, one obtains lower bounds on $|x_k(E+i\varepsilon)|$ which are uniform for $E \in [-K,K] \subseteq \R$. Namely, given $\varepsilon > 0$, choose $k$ minimal with the property $C_\delta' F_k^{- s'} < \varepsilon$. By \eqref{imwidth}, we infer that $|x_k(E+i\varepsilon)| > 1 + \delta$ and $|x_{k+1}(E+i\varepsilon)| > 1 + \delta$. Since $|x_{-1}(E+i\varepsilon)| = 1 \le 1 + \delta$, we must have the situation of \cite[Lemma~4]{DT07} (as recalled above) for some $K \le k$. In particular, for $k' > k$, \eqref{expgrowth} shows that
$$
|x_{k'}(E+i\varepsilon)| \ge (1 + \delta)^{F_{k'-k}}.
$$

This motivates the following definitions. Fix some small $\delta > 0$. For $T > 1$, denote by $k(T)$ the unique integer with
$$
\frac{F_{k(T) - 1}^{s'}}{C_\delta'} \le T < \frac{F_{k(T)}^{s'}}{C_\delta'}
$$
and let
$$
N(T) = F_{k(T) + \lfloor \sqrt{k(T)} \rfloor}.
$$
Thus, for every $\tilde \nu > 0$, there is a constant $C_{\tilde \nu} > 0$ such that
\begin{equation}\label{ntfib}
N(T) \le C_{\tilde \nu} T^\frac{1}{s'} T^{\tilde \nu}.
\end{equation}
It follows from \cite[Theorem~7]{DT07} and the argument above that\footnote{This estimate obviously works for $\omega = 0$ since then the trace and the norm are directly related. For general $\omega$, one can use the arguments developed in \cite{D05}. The central idea is that the trace of words of length $F_k$ occurring in the Fibonacci sequence is the same for all but one word and is given by $2x_k$. If the word in question is the ``bad'' one, we can simply shift by one to see a good word, derive the estimate there and divide by $C^2$, where $C$ bounds the norm of a one-step transfer matrix.}
\begin{align*}
\langle P(N(T),\cdot) \rangle (T) & \lesssim \exp (-c N(T)) + T^3 \int_{-K}^K \left( \max_{3 \le n \le N(T)} \left\| M \left( n; \omega, E+\tfrac{i}{T} \right) \right\|^2 \right) ^{-1} dE \\
& \lesssim \exp (-c N(T)) + T^3 (1 + \delta)^{-2 F_{\lfloor \sqrt{k(T)} \rfloor}}.
\end{align*}
(We can estimate the norm on the left half-line in a completely analogous way.) From this bound, we see that $\langle P(N(T),\cdot) \rangle (T)$ goes to zero faster than any inverse power of $T$. Therefore we can apply \cite[Theorem~1]{DT07} and obtain from \eqref{ntfib} that
$$
\tilde \alpha_u^+ \le \frac{1}{s'} + \tilde \nu = \left( \frac{\liminf_{k \to \infty} \frac1k \log \frac{1}{R_k(\delta)}}{\log \varphi} - \rho' \right)^{-1} + \tilde \nu.
$$
Since we can take $\rho' > 0$ and $\tilde \nu > 0$ arbitrarily small, \eqref{e.alphauupperbound} follows.

\bigskip

(b) Let $\lambda > 0$ and choose $\delta \in (0,\delta(\lambda)/2)$. Fix $k$ and $j$, and consider the connected component $B_k^{(j)}(2\delta)$ of $\sigma_k^{2\delta}$. Since $B_k^{(j)}(2\delta)$ contains exactly one zero of $x_k$, it follows from the maximum modulus principle and Rouch\'e's Theorem that
$$
x_k : \mathrm{int}(B_k^{(j)}(2\delta)) \to B(0,1 + 2\delta)
$$
is univalent, and hence
$$
x_k^{-1} : B(0,1 + 2\delta) \to \mathrm{int}(B_k^{(j)}(2\delta))
$$
is well-defined and univalent as well. Consequently, the following mapping is a Schlicht function:
$$
F : B(0,1) \to \C, \quad F(z) = \frac{x_k^{-1} ((1 + 2\delta)z) - E_k^{(j)}}{(1 + 2\delta) [(x_k^{-1})'(0)]}.
$$
That is, $F$ is a univalent function on $B(0,1)$ with $F(0) = 0$ and $F'(0) = 1$.

The Koebe Distortion Theorem (see \cite[Theorem~7.9]{c}) implies that
\begin{equation}\label{koebe}
\frac{|z|}{(1 + |z|)^2} \le |F(z)| \le \frac{|z|}{(1 - |z|)^2} \text{ for } |z| \le 1.
\end{equation}
Evaluate the bound \eqref{koebe} on the circle $|z| = \frac{1 + \delta}{1 + 2\delta}$. For such $z$, we obtain
$$
\frac{(1 + \delta)(1 + 2\delta)}{(2 + 3 \delta)^2} \le |F(z)| \le \frac{(1 + \delta)(1 + 2\delta)}{\delta^2}.
$$
By definition of $F$ this means that
$$
| x_k^{-1} ((1 + 2\delta)z) - E_k^{(j)} | \le \frac{(1 + \delta)(1 + 2\delta)}{\delta^2} (1 + 2\delta) |(x_k^{-1})'(0)|
$$
and
$$
| x_k^{-1} ((1 + 2\delta)z) - E_k^{(j)} | \ge \frac{(1 + \delta)(1 + 2\delta)}{(2 + 3\delta)^2} (1 + 2\delta) |(x_k^{-1})'(0)|
$$
for all $z$ with $|z| = \frac{1 + \delta}{1 + 2\delta}$. In other words, if $|z| = 1 + \delta$, then
\begin{equation}\label{variation}
| x_k^{-1} (z) - E_k^{(j)} | \le \frac{(1 + \delta)(1 + 2\delta)^2}{\delta^2} |(x_k^{-1})'(0)|
\end{equation}
and
\begin{equation}\label{variation2}
| x_k^{-1} (z) - E_k^{(j)} | \ge \frac{(1 + \delta)(1 + 2\delta)^2}{(2 + 3\delta)^2} |(x_k^{-1})'(0)|.
\end{equation}
Note that as $z$ runs through the circle of radius $1 + \delta$ around zero, the point $x_k^{-1} (z)$ runs through the entire boundary of $B_k^{(j)}(\delta)$. Thus, since $|(x_k^{-1})'(0)| = |x_k'(E_k^{(j)})|^{-1}$, \eqref{variation} and \eqref{variation2} yield
$$
B \Big( E_k^{(j)}, \frac{(1 + \delta)(1 + 2 \delta)^2}{(2 + 3\delta)^2} |x_k'(E_k^{(j)})|^{-1} \Big) \subseteq B_k^{(j)}(\delta) \subseteq B \Big( E_k^{(j)}, \left(\frac{(1 + \delta)(1 + 2\delta)}{\delta} \right)^2 |x_k'(E_k^{(j)})|^{-1} \Big).
$$
In particular, it follows that
$$
\frac{(1 + \delta)(1 + 2 \delta)^2}{(2 + 3\delta)^2} |x_k'(E_k^{(j)})|^{-1} \le r_k^{(j)}(\delta) \le R_k^{(j)}(\delta) \le \left(\frac{(1 + \delta)(1 + 2\delta)}{\delta} \right)^2 |x_k'(E_k^{(j)})|^{-1}.
$$
Thus,
$$
\left(\frac{\delta}{(1 + \delta)(1 + 2\delta)} \right)^2 |x_k'(E_k^{(j)})| \le \frac{1}{R_k^{(j)}(\delta)} \le \frac{1}{r_k^{(j)}(\delta)} \le \frac{(2 + 3\delta)^2}{(1 + \delta)(1 + 2 \delta)^2} |x_k'(E_k^{(j)})|,
$$
which in turn implies
$$
\left(\frac{\delta}{(1 + \delta)(1 + 2\delta)} \right)^2 \left( \min_j |x_k'(E_k^{(j)})| \right) \le \frac{1}{R_k(\delta)} \le \frac{1}{r_k(\delta)} \le \frac{(2 + 3\delta)^2}{(1 + \delta)(1 + 2 \delta)^2} \left( \min_j |x_k'(E_k^{(j)})| \right).
$$
This shows \eqref{e.rkRkxprimeconnection1}--\eqref{e.rkRkxprimeconnection2}.

\bigskip

(c) The estimates in this part follow immediately from the estimates in parts (a) and (b). This concludes the proof.
\end{proof}

\begin{proof}[Proof of \eqref{e.transportexponentidentity} in Theorem \ref{t.identities}.]
The identity is a direct consequence of Propositions~\ref{p.xkprimelyap} and \ref{p.transportbounds}.
\end{proof}

\section{The Density of States Measure}

In this section we discuss the density of states measure $\nu_\lambda$. Specifically, we establish the identity \eqref{e.doesmeasureidentity} and the large coupling asymptotics \eqref{e.doesmeasureasymptotics}.

The identity \eqref{e.doesmeasureidentity} was established in \cite{DG12} for $\lambda > 0$ sufficiently small. An inspection of the proof given there shows that all that is needed to extend the identity to all $\lambda > 0$ is the transversality statement provided by Theorem~\ref{t.transversal}. Thus, given that Theorem~\ref{t.transversal} has now been established, the identity \eqref{e.doesmeasureidentity} for all $\lambda > 0$ follows as an immediate consequence.

Proving \eqref{e.doesmeasureasymptotics} will require significantly more work. We begin with the following alternative identity for $\dim_H \nu_\lambda$, which we can prove for $\lambda$ sufficiently large. Recall that each connected component of $\sigma_k$ contains precisely one zero of $x_k$, denoted by $E_k^{(i)}$, $1 \le i \le F_k$.

\begin{prop}\label{p.dimnu}
For every $\lambda >0$ we have
\begin{equation}\label{e.dimnuzeros}
\dim_H \nu_\lambda = \frac{\log \varphi}{\lim_{k \to \infty} \frac{1}{kF_k} \log \left( \prod_{i = 1}^{F_k} \left| x_k'(E_k^{(i)}) \right| \right)}.
\end{equation}
\end{prop}

\begin{proof}
Due to \eqref{e.doesmeasureidentity}, we need to show that
$$
\mathrm{Lyap}^u \left( \mu_{\lambda,\mathrm{max}} \right) = \lim_{k \to \infty} \frac{1}{kF_k} \log \left( \prod_{i=1}^{F_k} \left| x_k'(E_k^{(i)}) \right| \right),
$$
which is equivalent to
$$
\mathrm{Lyap}^u \left( \mu_{\lambda,\mathrm{max}} \right) = \lim_{k \to \infty} \frac{1}{k F_{k-1}} \log \left( \prod_{i=1}^{F_{k-1}} \left| x_{k-1}'(E_{k-1}^{(i)}) \right| \right).
$$

Recall that $T^{k}_\lambda(\ell_\lambda(E)) = (x_{k+1}(E), x_k(E), x_{k-1}(E))$, and hence the $z$-component of $T^{k}_\lambda(\ell_\lambda(E))$ is $x_{k-1}(E)$. Let $l_i \in \ell_\lambda$ be the points such that $\ell_\lambda(E_{k-1}^{(i)}) = l_i$. Due to the uniform in $k$ transversality of $T^{k}_\lambda(\ell_\lambda)$  to the plane $\{z=0\}$ (which follows from Proposition \ref{p.trans} combined with the Inclination Lemma), we have $C^{-1} \|DT^{k}_\lambda(\overline{v_i})\| \le \left| x_{k-1}'(E_{k-1}^{(i)}) \right| \le C \|DT^{k}_\lambda(\overline{v_i})\|$ for some uniform $C > 1$, where $\overline{v_i}$ is a unit vector tangent to $\ell_\lambda$ at the point $l_i$. Therefore the statement can be reduced to the claim that
\begin{equation}\label{e.eschelyap}
\mathrm{Lyap}^u \left( \mu_{\lambda,\mathrm{max}} \right) = \lim_{k \to \infty} \frac{1}{kF_{k-1}} \log \left( \prod_{\{ l_i \in \ell_\lambda : T^{k}_\lambda(l_i) \in \{z=0\} \}} \left\| DT^{k}_\lambda(\overline{v_i}) \right\| \right).
\end{equation}
We will need the following statement from hyperbolic dynamics.

\begin{lemma}\label{l.nu}
Let $f : M^2 \to M^2$ be a $C^2$-diffeomorphism such that $f(\Lambda) = \Lambda$ is a topologically mixing locally maximal totally disconnected hyperbolic set, $\bigcap_{n \in \mathbb{Z}}f^n(U(\Lambda))=\Lambda$. Let $\gamma_1, \gamma_2\subset U$ be such that $\gamma_1$ is transversal to $W^s(\Lambda)$, and $\gamma_2$ is transversal to $W^u(\Lambda)$. For each $k\in \mathbb{N}$ denote by $\{l_i\}_{i=1, \ldots, N_k}\subset \gamma_1$ the set $f^{-k}(f^k(\gamma_1)\cap \gamma_2)$. Then,
\begin{equation}\label{e.lyap}
\mathrm{Lyap}^u(\mu_\mathrm{max}) = \lim_{k \to \infty} \frac{1}{kN_k} \log \left( \prod_{\{l_i \in \gamma_1 : f^k(l_i) \in \gamma_2 \}} \left| Df^k(\overline{v_i}) \right| \right),
\end{equation}
where $\mu_\mathrm{max}$ is the measure of maximal entropy for $f|_{\Lambda} : \Lambda \to \Lambda$, and $\overline{v}_i$ is a unit vector tangent to $\gamma_1$ at the point $l_i$.
\end{lemma}

\begin{proof}
First of all, let us notice that if $\gamma_1$ is represented as a disjoint union of curves $\gamma_1'$ and $\gamma_1''$, and (\ref{e.lyap}) holds for both $\gamma_1'$ and $\gamma_1''$, then it also holds for the initial curve $\gamma_1$. Indeed, this just follows from the fact that if $\{a_n\}$, $\{b_n\}$, $\{x_n\}$, and $\{y_n\}$ are sequences of positive numbers such that $\frac{a_n}{b_n}\to c$ and $\frac{x_n}{y_n}\to c$ then $\frac{a_n+x_n}{b_n+y_n}\to c$. The same statement (due to the same argument) holds for the curve $\gamma_2$.

Next, let us notice that if \eqref{e.lyap} holds for some $\gamma_1$, then it also holds for $f(\gamma_1)\cap U$ (and vice versa). Indeed, $N_k(\gamma_1)=N_{k-1}(f(\gamma_1)\cap U)$, and the expression $\log \left( \prod_{\{ l_i \in \gamma_1 : f^k(l_i) \in \gamma_2 \}} \left| Df^k(\overline{v_i}) \right| \right)$ differs from the expression $\log \left( \prod_{\{ l_i \in f(\gamma_1) : f^{k-1}(l_i) \in \gamma_2 \}} \left| Df^k(\overline{v_i}) \right| \right)$ by no more than $\mathrm{const} \cdot N_k(\gamma_1)$. Combining these two observations, we see that it is enough to prove \eqref{e.lyap} for the case when $\gamma_1$ is a curve that is $C^1$-close to a piece of  unstable manifold of $\Lambda$ inside a rectangle of a Markov partition, and $\gamma_2$ is a curve that is $C^1$-close to a piece of stable manifold of $\Lambda$ inside a rectangle of a Markov partition.

Moreover, we can further reduce the statement to the case when $\gamma_1$ is a piece of an unstable manifold in some element of Markov partition, and $\gamma_2$ is a piece of a stable manifold in some element of Markov partition. Indeed, let us consider $C^1$-invariant stable and unstable foliations in $U(\Lambda)$ that include stable and unstable laminations $W^s(\Lambda)$ and $W^u(\Lambda)$ and the curves $\gamma_1$ and $\gamma_2$, respectively. For the existence of these foliations, see \cite{W}. Expansion of the differential of $f$ along the line tangent to a leaf of the unstable foliation is a $C^1$-function. Exponential instability of orbits near a hyperbolic set (see Proposition 6.4.16 from \cite{KH}) now implies that if (\ref{e.lyap}) holds for pieces of stable and unstable manifolds as $\gamma_1$ and $\gamma_2$, then it also holds for the initial curves $\gamma_1$ and $\gamma_2$ that were sufficiently $C^1$-close to the pieces of stable and unstable manifolds.

From now on we can assume that $\gamma_1$ is a piece of an unstable manifold in some element of Markov partition, and $\gamma_2$ is a piece of a stable manifold in some element of Markov partition. The restriction $f|_{\Lambda}$ is conjugate to a topological Markov shift $\sigma_A : \Sigma_A \to \Sigma_A$ with some transitive $0-1$ matrix $A$ of size $N \times N$.

\begin{lemma}\label{l.convnu}
Let $\sigma_A : \Sigma_A \to \Sigma_A$ be a transitive topological Markov chain and denote by $\nu_P$ the measure of maximal entropy {\rm (}Parry measure{\rm )}. Fix any $\omega', \omega'' \in \{1, 2, \ldots, N\}$, and admissible sequences

$\underline{\ldots \text{sequence}_1}\,\, \omega'$ \ -- \ infinite to the left, and

$\omega''\, \underline{\text{sequence}_2\ldots}$  \ -- \ infinite to the right.

We assume that at least one of the two one-sided sequences is not eventually periodic.

For each $k \in \mathbb{N}$, consider the collection $X_k$ of all the sequences from $\Sigma_A$ of the form
$$
\underline{\ldots \text{sequence}_1}\,\, \underbrace{\stackrel{*}{\omega'} \ldots \ldots \omega''}_{k+1}\, \underline{\text{sequence}_2\ldots}
$$
{\rm (}where $*$ indicates the origin{\rm )} and set $S_k = \bigcup_{j=0}^{k-1} \sigma_A^j(X_k)$. Then
$$
\nu_k := \frac{1}{\# S_k} \sum_{x \in S_k} \delta_x \to \nu_P\ \ \ \text{as} \ k\to \infty.
$$
\end{lemma}

Notice that Lemma~\ref{l.convnu} immediately implies \eqref{e.lyap} in the case when $\gamma_1$ and $\gamma_2$ are pieces of stable and unstable manifolds. Indeed, we can assume without loss of generality that $\gamma_1, \gamma_2$ do not contain any periodic points (otherwise we deform them slightly). Let $H : \Sigma_A \to \Lambda$ be the conjugacy between $\sigma_A$ and $f|_{\Lambda}$. Then $H_* (\nu_P) = \mu_\mathrm{max}$, and if $\phi : \Lambda \to \R$ is a continuous function, then
$$
\int \phi \, d\left(H_*(\nu_k)\right) = \frac{1}{\# S_k} \sum_{x \in S_k} \phi(H(x)) \to \int \phi \, d\mu_{max},
$$
and hence for $\phi(x) = \log |Df_x(\bar v_x)|$, where $\bar v_x$ is a unit vector tangent to a leaf of the unstable foliation at the point $x$, we have
\begin{align*}
\int \phi \, d\left(H_*(\nu_k)\right) & = \frac{1}{kN_k} \log \left( \prod_{\{l_i \in \gamma_1 : f^k(l_i) \in \gamma_2\}} \left| Df^k(\overline{v_i}) \right| \right) \\
& \to \int \phi \, d\mu_{max} \\
& = \mathrm{Lyap}^u(\mu_{max})
\end{align*}
as $k\to \infty$, and therefore \eqref{e.lyap} holds.

\begin{proof}[Proof of Lemma \ref{l.convnu}]
First of all, let us recall the construction of the Parry measure $\nu_P$. Due to the Perron-Frobenius Theorem, the matrix $A = (A_{ij})$ has only one eigenvector $\bar{v} = (v_1, \ldots, v_N)$ with positive entries. The eigenvalue $\lambda > 1$ that corresponds to $\bar{v}$ is larger than the absolute value of any other eigenvalue of $A$. Denote by $\bar u=(u_1, \ldots, u_N)$ the eigenvector of the transposed matrix $A^T$ that corresponds to the eigenvalue $\lambda$. Without loss of generality we can normalize $\bar v$ and $\bar u$ in such a way that $v_1 u_1 + v_2 u_2 + \ldots +v_N u_N = 1$. The Parry measure is the Markov measure with the stationary probability vector $\bar p = (p_1, \ldots, p_N)$, $p_i = v_iu_i$, and the transition matrix $(p_{ij})$, $p_{ij} = \frac{A_{ij} v_j}{\lambda v_j}$. An equivalent way to introduce the Parry measure is to define it on a cylinder $C = \{\omega \in \Sigma_A : \omega_0 = i_0, \ldots, \omega_n = i_n\}$ by
$$
\nu_P(C)=\left\{
           \begin{array}{cl}
             0 & \hbox{if $i_0\ldots i_n$ is not an admissible sequence;} \\
             \frac{u_{i_0}v_{i_n}}{\lambda^n} & \hbox{if $i_0\ldots i_n$ is admissible.}
           \end{array}
         \right.
$$
We need to show that for any continuous function $\phi : \Sigma_A \to \mathbb{R}$, we have
\begin{equation}\label{e.conv}
\int \phi \, d\nu_k = \frac{1}{\# S_k} \sum_{x \in S_k} \phi(x) \to \int \phi \, d \nu_P \ \ \text{as} \ k\to \infty.
\end{equation}
It is enough to establish this convergence for functions of the form $\phi_C = \chi_C$, where $C = \{ \omega \in \Sigma_A : \omega_r = i_r, \omega_{r+1} = i_{r+1}, \ldots, \omega_s = {i_s}\}$ for some $r < s$ and $i_j \in \{1, 2, \ldots, N\}$ since the linear combinations of these functions are dense in $C(\Sigma_A)$.

\begin{lemma}\label{l.enc}
Consider a topological Markov chain $\sigma_A : \Sigma_A \to \Sigma_A$ and fix some finite admissible sequence $[i_0, i_1, \ldots, i_t]$, $i_j \in \{1, \ldots, N\}$, and $\omega', \omega'' \in \{1, \ldots, N\}$. For a given $k \in \mathbb{N}$, consider the collection of all admissible sequences $\omega_0, \ldots, \omega_k$ of length $k+1$ such that $\omega_0 = \omega'$ and $\omega_k = \omega''$, and denote by $I_{[i_0, i_1, \ldots, i_t]}(\omega', \omega'')$ the number of times the string $[i_0, i_1, \ldots, i_t]$ can be encountered in these sequences {\rm (}counting different encounters in the same sequence as separate{\rm )}. If we denote $A^k = (A^{(k)}_{ij})$, then
$$
\frac{I_{[i_0, i_1, \ldots, i_t]}(\omega', \omega'')}{kA_{\omega'\omega''}^{(k)}} \to \frac{u_{i_0}v_{i_t}}{\lambda^t} \ \  \ \text{as} \ k \to \infty.
$$
\end{lemma}

\begin{proof}
Let us take $k \gg t$ and represent
$$
I_{[i_0, i_1, \ldots, i_t]}(\omega', \omega'') = I_{[i_0, i_1, \ldots, i_t]}^{\mathrm{bound}}(\omega', \omega'')+I_{[i_0, i_1, \ldots, i_t]}^{\mathrm{int}}(\omega', \omega''),
$$
where $I_{[i_0, i_1, \ldots, i_t]}^{\mathrm{bound}}(\omega', \omega'')$ is the number of encounters of $[i_0, i_1, \ldots, i_t]$ starting in the beginning or in the tail part of length $[\ln k]$ of the sequences $\omega_0, \ldots, \omega_k$, and $I_{[i_0, i_1, \ldots, i_t]}^{\mathrm{int}}(\omega', \omega'')$ is the number of encounters of $[i_0, i_1, \ldots, i_t]$ starting in the middle part (of length $k-2[\ln k]$) of these sequences. A rough estimate on $I_{[i_0, i_1, \ldots, i_t]}^{\mathrm{bound}}(\omega', \omega'')$ gives
$$
I_{[i_0, i_1, \ldots, i_t]}^{\mathrm{bound}}(\omega', \omega'') \le C k^{\ln N} \ln k,
$$
and hence in
$$
\frac{I_{[i_0, i_1, \ldots, i_t]}(\omega', \omega'')}{kA_{\omega'\omega''}^{(k)}} = \frac{I_{[i_0, i_1, \ldots, i_t]}^{\mathrm{bound}}(\omega', \omega'')}{kA_{\omega'\omega''}^{(k)}} + \frac{I_{[i_0, i_1, \ldots, i_t]}^{\mathrm{int}}(\omega', \omega'')}{kA_{\omega'\omega''}^{(k)}},
$$
we have $\frac{I_{[i_0, i_1, \ldots, i_t]}^{\mathrm{bound}}(\omega', \omega'')}{kA_{\omega'\omega''}^{(k)}}\to 0$ as $k\to \infty$.

For a given $l$ between $[\ln k]$ and $k-[\ln k]$, denote by $I^l$ the number of admissible sequences $\omega_0, \ldots, \omega_k$ such that $[\omega_l\omega_{l+1}\ldots \omega_{l+t}]=[i_0i_1\ldots i_t]$. We have
\begin{align*}
\frac{I^l}{A^{(k)}_{\omega'\omega''}} & = \sum_{\stackrel{\omega_0 = \omega', \omega_k = \omega'',}{[\omega_l\omega_{l+1}\ldots \omega_{l+t}]=[i_0i_1\ldots i_t]}} \frac{A_{\omega_0\omega_1} A_{\omega_1\omega_2} \ldots A_{\omega_{k-1}\omega_k}}{A^{(k)}_{\omega'\omega''}} \\
& = (A_{i_0i_1}A_{i_1i_2}\ldots A_{i_{t-1}i_t})\frac{A^{(l)}_{\omega'i_0}A^{(k-l-t)}_{i_t\omega''}}{A^{(k)}_{\omega'\omega''}}.
\end{align*}
Notice that $A_{i_0 i_1} A_{i_1 i_2} \ldots A_{i_{t-1} i_t} = 1$ since $i_0 i_1 \ldots i_t$ is an admissible sequence. We also know that $\lim_{k \to \infty} A^{(k)}_{i j} \lambda^{-k} = u_j v_i$ (see, e.g., \cite[Theorem 0.17]{W}). Since there are only finitely many pairs $(i j)$, the limit here is uniform in $i,j$, and therefore we have
$$
\frac{I^l}{A^{(k)}_{\omega' \omega''}} = \frac{\left( A^{(l)}_{\omega'i_0} \lambda^{-l} \right) \left( A^{(k-l-t)}_{i_t \omega''} \lambda^{-(k-l-t)} \right)}{A^{(k)}_{\omega' \omega''} \lambda^{-k} \lambda^t} \approx \frac{u_{i_0} v_{\omega'} \cdot u_{\omega''} v_{i_t}}{u_{\omega''} v_{\omega'} \lambda^t} = \frac{u_{i_0} v_{i_t}}{\lambda^t},
$$
uniformly for large $k$, and hence
$$
\frac{I_{[i_0, i_1, \ldots, i_t]}^{\mathrm{int}}(\omega', \omega'')}{k A_{\omega'\omega''}^{(k)}} = \frac{1}{k} \sum_{l = [\ln k]}^{k - \ln[k]} \frac{I^l}{A_{\omega' \omega''}^{(k)}} \to \frac{u_{i_0} v_{i_t}}{\lambda^t} \ \text{as}\ k \to \infty.
$$
This proves Lemma~\ref{l.enc}.
\end{proof}

Notice that Lemma~\ref{l.enc} implies \eqref{e.conv} for the function $\phi_C$. Indeed, if $\ln k\gg \max(|s|, |r|)$, then
$$
I_{[i_r, i_1, \ldots, i_s]}^{\mathrm{int}}(\omega', \omega'') \le \sum_{x \in S_k} \phi_C(x) \le I_{[i_r, i_1, \ldots, i_s]}^{\mathrm{int}}(\omega', \omega'') + C k^{\ln N} \ln k,
$$
and \eqref{e.conv} follows since $\# S_k = k A_{\omega'\omega''}^{(k)}$ by the assumption that at least one of the one-sided sequences ($\underline{\ldots \text{sequence}_1}\,\, \omega'$, $\omega''\, \underline{\text{sequence}_2\ldots}$) is not eventually periodic. This proves Lemma~\ref{l.convnu}.
\end{proof}

This concludes the proof of Lemma~\ref{l.nu}.
\end{proof}

Now \eqref{e.eschelyap} follows directly from Lemma~\ref{l.nu}, and this proves Proposition~\ref{p.dimnu}.
\end{proof}

For $\lambda$ sufficiently large, the modulus of $x_k'(E_k^{(i)})$ may be estimated with the help of \cite[Lemmas~5\,{\&}\,6]{DEGT}. Namely, if $m$ denotes the number of spectra $\sigma_j$, $1 \le j \le k-1$, $E_k^{(i)}$ belongs to, then
\begin{equation}\label{e.xkeiestimate}
S_l(\lambda)^m \le |x_k'(E_k^{(i)})| \le S_u(\lambda)^m,
\end{equation}
where
\begin{equation}\label{e.slulambda}
S_l(\lambda) = \frac{1}{2} \left((\lambda - 4) + \sqrt{(\lambda - 4)^2 - 12} \right) \quad \text{and} \quad S_u(\lambda) = 2\lambda + 22.
\end{equation}
Here, the first inequality in \eqref{e.xkeiestimate} requires $\lambda \ge 8$ and the second requires $\lambda > 4$.

Through the end of this section let us assume that $\lambda > 4$. In this case the Fricke-Vogt invariant implies that
\begin{equation}\label{e.critical}
\sigma_k \cap \sigma_{k+1} \cap \sigma_{k+2} = \emptyset.
\end{equation}
The identity \eqref{e.critical} is the basis for work done by Raymond \cite{Ra}. Following \cite{kkl}, we call a band $I_k \subset \sigma_k$ a ``type~A band'' if $I_k \subset \sigma_{k-1}$ (and hence $I_k \cap (\sigma_{k+1} \cup \sigma_{k-2}) = \emptyset$). We call a band $I_k \subset \sigma_k$ a ``type~B band'' if $I_k \subset \sigma_{k-2}$ (and therefore $I_k \cap \sigma_{k-1} = \emptyset$). Then we have the following result (Lemma~5.3 of \cite{kkl}, essentially Lemma~6.1 of \cite{Ra}).

\begin{lemma}\label{l.order}
For every $\lambda > 4$ and every $k \ge 1$, \\[1mm]
{\rm (a)} Every type A band $I_k \subset \sigma_k$ contains exactly one type B band $I_{k+2} \subset \sigma_{k+2}$, and no other bands from $\sigma_{k+1}$, $\sigma_{k+2}$. \\[1mm]
{\rm (b)} Every type B band $I_k \subset \sigma_k$ contains exactly one type A band $I_{k+1} \subset \sigma_{k+1}$ and two type B bands from $\sigma_{k+2}$, positioned around $I_{k+1}$.
\end{lemma}

We denote by $a_k$ the number of bands of type A in $\sigma_k$ and by $b_k$ the number of bands of type B in $\sigma_k$. By Raymond's work, it follows immediately that $a_k + b_k = F_k$ for every $k$. In fact, we have the following result, which follows from Lemma~\ref{l.order} by an easy induction.

\begin{lemma}
The constants $\{a_k\}$ and $\{b_k\}$ obey the relations
\begin{equation}\label{abkrecursion}
a_k = b_{k-1} , \quad b_k = a_{k-2} + 2b_{k-2}
\end{equation}
with initial values $a_0 = 1$, $a_1 = 0$, $b_0 = 0$, and $b_1 = 1$. Consequently, for $k \ge 2$,
\begin{equation}\label{abkvalues}
a_k = b_{k-1} = F_{k-2}.
\end{equation}
\end{lemma}

Let us also denote by $a_{k,m}$ the number of bands $b$ of type A in $\sigma_k$ with $\# \{ 0\le j < k : b \cap \sigma_j \not= \emptyset \} = m$ and by $b_{k,m}$ the number of bands $b$ of type B in $\sigma_k$ with $\# \{ 0 \le j < k : b \cap \sigma_j \not= \emptyset \} = m$. Then, \cite[Lemma~4]{DEGT} reads as follows:

\begin{lemma}
We have
\begin{equation}\label{e.abkmrecursion}
a_{k,m} = b_{k-1,m-1} , \quad b_{k,m} = a_{k-2,m-1} + 2b_{k-2,m-1}
\end{equation}
with initial values $a_{0,m} = 0$ for $m > 0$, $a_{0,0} = 1$, $a_{1,m} = 0$ for $m \ge 0$, $b_{0,m} = 0$ for $m \ge 0$, $b_{1,m}
= 0$ for $m > 0$, and $b_{1,0} = 1$. Consequently,
\begin{equation}\label{e.abkmzero}
a_{k,m} = b_{k-1,m-1} = \begin{cases} 2^{2k - 3m - 1} \frac{m}{k-m} { k - m \choose 2m - k } & \text{ when } \lceil \tfrac{k}{2} \rceil
\le m \le \lfloor \tfrac{2k}{3} \rfloor; \\ 0 & \text{ otherwise.}
\end{cases}
\end{equation}
\end{lemma}

In fact, for our purposes here the recursion \eqref{e.abkmrecursion} will be sufficient, and we won't make use of the explicit solution \eqref{e.abkmzero}. Verifying the recursion \eqref{e.abkmrecursion} using the definition and Lemma~\ref{l.order} is straightforward.

Set
$$
A_k = \sum_m m a_{k,m}, \quad B_k = \sum_m m b_{k,m}, \quad \text{and} \quad C_k = A_k + B_k.
$$

\begin{lemma}
We have
\begin{align}
A_k & = B_{k-1} + F_{k-2}, \label{e.Arec} \\
B_k & = A_{k-2} + 2 B_{k-2} + F_{k-1}, \label{e.Brec} \\
C_k & = C_{k-1} + C_{k-2} + 2F_{k-2}. \label{e.Crec}
\end{align}
\end{lemma}

\begin{proof}
We have
\begin{align*}
A_k & = \sum_m m a_{k,m} \\
& = \sum_m m b_{k-1,m-1} \\
& = \sum_m (m - 1 + 1) b_{k-1,m-1} \\
& = B_{k-1} + b_{k-1} \\
& = B_{k-1} + F_{k-2}.
\end{align*}
Here we used \eqref{e.abkmrecursion} in the second step, \eqref{abkvalues} in the fifth step, and the definitions in the other steps. This establishes \eqref{e.Arec}.

Similarly, we have
\begin{align*}
B_k & = \sum_m m b_{k,m} \\
& = \sum_m m (a_{k-2,m-1} + 2b_{k-2,m-1}) \\
& = \sum_m (m - 1 + 1) a_{k-2,m-1} + 2 \sum_m (m - 1 + 1) b_{k-2,m-1} \\
& = A_{k-2} + a_{k-2} + 2 B_{k-2} + 2 b_{k-2} \\
& = A_{k-2} + F_{k-4} + 2 B_{k-2} + 2 F_{k-3} \\
& = A_{k-2} + 2 B_{k-2} + F_{k-1}.
\end{align*}
Here we used \eqref{e.abkmrecursion} in the second step, \eqref{abkvalues} in the fifth step, the Fibonacci number recursion twice in the sixth step, and the definitions in the other steps. This establishes \eqref{e.Brec}.

Finally, we have
\begin{align*}
C_k & = A_k + B_k \\
& = B_{k-1} + F_{k-2} + A_{k-2} + 2 B_{k-2} + F_{k-1} \\
& = B_{k-1} + C_{k-2} + B_{k-2} + F_{k} \\
& = C_{k-1} - A_{k-1} + C_{k-2} + B_{k-2} + F_{k} \\
& = C_{k-1} - F_{k-3} + C_{k-2} + F_{k} \\
& = C_{k-1} + C_{k-2} + 2F_{k-2}.
\end{align*}
Here we used \eqref{e.Arec} and \eqref{e.Brec} in the second step, the definition and the Fibonacci number recursion in the third step, \eqref{e.Arec} in the fourth step, and the Fibonacci number recursion twice in the sixth step. This establishes \eqref{e.Crec}.
\end{proof}

\begin{prop}\label{p.combinatoriallimit}
We have
\begin{equation}\label{e.ckkfklimit}
\lim_{k \to \infty} \frac{C_k}{k F_k} = \frac{4}{5 + \sqrt{5}}.
\end{equation}
In particular,
$$
\frac{\log \varphi}{\lim_{k \to \infty} \frac{C_k}{k F_k}} = \frac{5 + \sqrt{5}}{4} \log \varphi \approx 1.80902 \log \varphi.
$$
\end{prop}

\begin{proof}
Set
$$
\beta = \frac{4}{5 + \sqrt{5}} = \frac{2}{\varphi + 2}
$$
and $R_k = C_k - \beta k F_k$. Then,
\begin{align*}
R_k - R_{k-1} - R_{k-2} & = C_k - C_{k-1} - C_{k-2} - \beta k F_k + \beta (k-1) F_{k-1} + \beta (k-2) F_{k-2} \\
& = 2F_{k-2} - \beta F_{k-1} - 2 \beta F_{k-2} \\
& = 2F_{k-2} \left( 1 - \frac{\beta}{2} \frac{F_{k-1}}{F_{k-2}} - \beta \right) \\
& = 2F_{k-2} \left( 1 - \frac{\beta}{2} \left( \varphi + \left( \frac{F_{k-1}}{F_{k-2}} - \varphi \right) \right) - \beta \right) \\
& = 2F_{k-2} \frac{\beta}{2} \left( \varphi - \frac{F_{k-1}}{F_{k-2}} \right) \\
& = \beta \left( F_{k-2} \varphi - F_{k-1} \right)
\end{align*}
For the fifth step, note that
$$
1 - \frac{1}{\varphi + 2} \varphi - \frac{2}{\varphi + 2} = 0.
$$
By a standard estimate from the theory of continued fractions, this shows that
\begin{equation}\label{e.Rkrecestimate}
|R_k - R_{k-1} - R_{k-2}| < \frac{\beta}{F_{k-1}}.
\end{equation}
Set $C = \max \{ |R_1| , |R_2| \}$ and apply \eqref{e.Rkrecestimate} repeatedly to obtain
\begin{align*}
|R_1| & \le C \\
|R_2| & \le C \\
|R_3| & < 2C + \frac{\beta}{F_{2}} \\
|R_4| & < 3C + \frac{\beta}{F_{2}} + \frac{\beta}{F_{3}} \\
|R_5| & < 5C + 2\frac{\beta}{F_{2}} + \frac{\beta}{F_{3}} + \frac{\beta}{F_{4}} \\
|R_6| & < 8C + 3\frac{\beta}{F_{2}} + 2 \frac{\beta}{F_{3}} + \frac{\beta}{F_{4}} + \frac{\beta}{F_{5}} \\
& \; \; \vdots \\
|R_k| & < F_{k-1} C + F_{k-2} \frac{\beta}{F_{2}} + F_{k-3} \frac{\beta}{F_{3}} + F_{k-4} \frac{\beta}{F_{4}} + \cdots + F_0 \frac{\beta}{F_{k}} \\
& = F_k \left( \frac{F_{k-1}}{F_k} C + \frac{F_{k-2}}{F_k} \frac{\beta}{F_{2}} + \frac{F_{k-3}}{F_k} \frac{\beta}{F_{3}} + \frac{F_{k-4}}{F_k} \frac{\beta}{F_{4}} + \cdots + \frac{F_0}{F_k} \frac{\beta}{F_{k}} \right).
\end{align*}
This implies $|R_k| = O(F_k)$, and in particular
$$
\lim_{k \to \infty} \frac{R_k}{k F_k} = 0.
$$
In view of $R_k = C_k - \beta k F_k$, this establishes \eqref{e.ckkfklimit} and concludes the proof of the proposition.
\end{proof}

We are now in a position to prove \eqref{e.doesmeasureasymptotics}. This result will be an easy consequence of Proposition~\ref{p.dimnu}, the estimates \eqref{e.xkeiestimate}, and Proposition~\ref{p.combinatoriallimit}.

\begin{proof}[Proof of \eqref{e.doesmeasureasymptotics}.]
By \eqref{e.dimnuzeros}, we have
$$
\dim_H \nu_\lambda = \frac{\log \varphi}{\lim_{k \to \infty} \frac{1}{kF_k} \log \left( \prod_{i = 1}^{F_k} \left| x_k'(E_k^{(i)}) \right| \right)}
$$
for $\lambda \ge 16$. By \eqref{e.xkeiestimate}, we have
$$
S_l(\lambda)^{m(E_k^{(i)})} \le |x_k'(E_k^{(i)})| \le S_u(\lambda)^{m(E_k^{(i)})},
$$
where $m(E_k^{(i)})$ denotes the number of spectra $\sigma_j$, $1 \le j \le k-1$, $E_k^{(i)}$ belongs to, and $S_l(\lambda)$, $S_u(\lambda)$ are given in \eqref{e.slulambda}. Thus,
$$
\log \left( \prod_{i = 1}^{F_k} \left| x_k'(E_k^{(i)}) \right| \right) = \sum_{i = 1}^{F_k} \log \left| x_k'(E_k^{(i)}) \right| = \sum_{m = \lceil \frac{k}{2} \rceil}^{\lfloor \frac{2k}{3} \rfloor} \sum_{m(E_k^{(i)}) = m} \log \left| x_k'(E_k^{(i)}) \right|,
$$
and hence
$$
\lim_{k \to \infty} \frac{1}{kF_k} \log \left( \prod_{i = 1}^{F_k} \left| x_k'(E_k^{(i)}) \right| \right) \in \left[ \frac{4}{5 + \sqrt{5}} \log S_l(\lambda) , \frac{4}{5 + \sqrt{5}} \log S_u(\lambda) \right]
$$
by \eqref{e.ckkfklimit} in Proposition~\ref{p.combinatoriallimit}. We obtain
$$
\lim_{\lambda \to \infty} \dim_H \nu_\lambda \cdot \log \lambda = \lim_{\lambda \to \infty} \frac{\log \varphi \cdot \log \lambda}{\lim_{k \to \infty} \frac{1}{kF_k} \log \left( \prod_{i = 1}^{F_k} \left| x_k'(E_k^{(i)}) \right| \right)} = \frac{5 + \sqrt{5}}{4} \log \varphi,
$$
which concludes the proof.
\end{proof}

\section{The Optimal H\"older Exponent}

In this section we provide an explicit expression for the optimal H\"older exponent of the integrated density of states for the Fibonacci Hamiltonian. It is based on the following dynamical result.

\begin{theorem}\label{t.tech}
Let $T : M^2 \to M^2$ be a $C^{1+\alpha}$-diffeomorphism with a {\rm (}topologically{\rm )} zero-dimensional basic set $\Lambda$, and $\mu_{max}$ be the measure of maximal entropy for $T_{\Lambda}$. Let $L \subset M$ be a smooth curve transversal to $W^s(\Lambda)$ with parametrization $L : \mathbb{R} \to M^2$ such that $L \cap W^s(\Lambda)$ is compact. Let $R$ be an element of a Markov partition for $\Lambda$, and let $\pi : \Lambda \cap R \to L$ be a continuous projection along the stable manifolds. Set $\nu = L^{-1} \circ \pi(\mu_{max}|_R)$, and denote by $\gamma$ the optimal H\"older exponent of $\nu$. Then,
$$
\gamma = \frac{h_{\mathrm{top}} (T|_{\Lambda})}{\sup_{p \in Per(T|_{\Lambda})} \mathrm{Lyap}^u(p)}.
$$
In other words, we have the following:

\begin{itemize}

\item[{\rm (1)}] For any $\gamma_0 < \gamma$ and any sufficiently small interval $I \subset \mathbb{R}$, we have $\nu(I) < |I|^{\gamma_0}$;

\item[{\rm (2)}] For any $\gamma_1 > \gamma$ and any $\varepsilon > 0$, there exists an interval $I \subset \mathbb{R}$ such that $|I| < \varepsilon$ and $\nu(I) > |I|^{\gamma_1}$.

\end{itemize}
\end{theorem}

\begin{proof}
Fix any $\gamma_0 \in (0,\gamma)$ and suppose that $I = [E_0,E_1] \subset \R$ is sufficiently small (we will determine the appropriate smallness condition later). Without loss of generality we can assume that $L(E_0) , L(E_1) \in W^s(\Lambda)$ (otherwise we can decrease the size of $I$ without changing its measure).

Consider the rectangle $R_I = \pi^{-1}(L(I)) \subset R$. Then, $\mu_\mathrm{max}(R_I) = \nu(I)$. Let $N \in \Z_+$ be the smallest value such that $T^N(R_I) \cap \Lambda$ is not a subset of just one element of the Markov partition (and hence has size of order one). We claim that
$$
C^{-1} \le \left| \frac{\mu_\mathrm{max}(R_I)}{e^{-N h_\mathrm{top}(T|_\Lambda)}} \right| \le C
$$
with $C$ uniform for all sufficiently small $I$. Indeed, consider the topological Markov chain $\sigma_A : \Sigma_A \to \Sigma_A$ conjugate to $T|_\Lambda : \Lambda \to \Lambda$. Then,
$$
\mu_\mathrm{max}(R_I) = \lim_{M \to \infty} \frac{\# (\mathrm{Fix} (T^M) \cap R_I)}{\# \mathrm{Fix} (T^M)}.
$$
But for large $M$ and $N$, we have
$$
\# \mathrm{Fix} (T^M) = \mathrm{Tr} (A^M)=  e^{M h_\mathrm{top}(T|_\Lambda)}(1+o(1)),
$$
since the largest eigenvalue of $A$ is equal to $e^{h_\mathrm{top}(T|_{\Lambda_\lambda})}$. At the same time the number of periodic orbits of period $M$ with prescribed initial segment of length $N < M$ is given by
$$
\quad \# (\mathrm{Fix} (T^M) \cap R_I) = e^{(M-N) h_\mathrm{top}(T|_\Lambda)} \cdot O(1),
$$
where $O(1)$ is bounded from above and away from zero uniformly in all $1 \ll N \ll M$, and hence $\mu_\mathrm{max}(R_I) = e^{-N h_\mathrm{top}(T|_\Lambda)} \cdot O(1)$.

On the other hand, $|I| = E_1 - E_0$ is of order of the width of $R_I$. Pick any point $p \in \Lambda \cap R_I$ and consider $W^u_\mathrm{loc}(p) \cap R_I$. Since a holonomy map along stable manifolds is $C^1$, we have
$$
|I| = | W^u_\mathrm{loc}(p) \cap R_I |\cdot O(1).
$$
The usual distortion argument shows also that
$$
| W^u_\mathrm{loc}(p) \cap R_I |= \frac{1}{\|DT^N(p)|_{E^u}\|} \cdot O(1),
$$
and hence for any $\varepsilon > 0$, $\varepsilon < \gamma - \gamma_0$,
\begin{align*}
\frac{\log \nu(I)}{\log |I|} & = \frac{-N h_\mathrm{top}(T|_\Lambda)+O(1)}{-\log \|DT^N(p)\|+O(1)} \\
& > \frac{h_\mathrm{top}(T|_\Lambda)}{\mathrm{Lyap}^u(p)} -\varepsilon \\
& \ge \frac{h_\mathrm{top}(T|_\Lambda)}{\sup_p \mathrm{Lyap}^u(p)} - \varepsilon \\
& = \gamma - \varepsilon>\gamma_0
\end{align*}
if $N$ is large enough (which can be guaranteed by choosing sufficiently small $|I|$). Therefore $\nu(I) < |I|^{\gamma_0}$.

\bigskip

Let us now take an arbitrary $\gamma_1 > \gamma$. There exists a periodic point $q \in R$ such that
$$
\frac{h_\mathrm{top}(T|_\Lambda)}{\mathrm{Lyap}^u(q)} < \gamma_1.
$$
For a given $\varepsilon \in (0, \gamma_1 - \gamma)$, consider a narrow rectangle $R_I \subset R$, $I \subset \R$, $R_I = \pi^{-1}(L(I))$, such that $|I| < \varepsilon$ and $q \in R_I$. If $N \in \Z_+$ is the smallest number such that $T^N(R_I)$ does not belong to one element of the Markov partition, then
$$
|I| = \frac{O(1)}{\|DT^N(q)|_{E^u_q}\|}
$$
and
$$
\nu(I) = \mu_\mathrm{max}(R_I) = e^{-N h_\mathrm{top}(T|_\Lambda)}\cdot O(1).
$$
Hence,
$$
\frac{\log \nu(I)}{\log |I|} = \frac{-N h_\mathrm{top}(T|_\Lambda)+O(1)}{-N(\frac1N \log \|DT^N(q)|_{E^u_q}\|+O(1)} \le \frac{h_\mathrm{top} (T|_\Lambda)}{\mathrm{Lyap}^u(q)} + \varepsilon < \gamma_1,
$$
and therefore $\nu(I) > |I|^{\gamma_1}$.
\end{proof}

\begin{proof}[Proof of \eqref{e.hoelderexponentidentity}.]
The theorem follows as a special case from Theorem~\ref{t.tech} since the density of states measure arises from the measure of maximal entropy for the trace map in the way required for Theorem~\ref{t.tech} to be applicable. This was shown in \cite{DG12} for small values of the coupling constant $\lambda$, and due to Theorem \ref{t.transversal} the same holds for all $\lambda > 0$.
\end{proof}

For small values of the coupling constant, $\sup_{p\in Per(f|_{\Lambda})}\mathrm{Lyap}^u(p)$ is attained in the periodic points {\rm (}of period $2$ and $6${\rm )} born from the singularities of the Cayley cubic, and therefore it can be calculated explicitly {\rm (}see, e.g., the proof of \cite[Lemma~3.3]{DG13}{\rm )}. Hence we get the following

\begin{coro}\label{c.holderexp}
For $\lambda > 0$ sufficiently small, we have
{\tiny
$$
\gamma_\lambda=\frac{2\log\left(\frac{\sqrt{5}+1}{2}\right)}{\log\left({{\frac{\sqrt{2} \sqrt{256 I^2+16 \left(2 A + 3 \sqrt{2} B + \sqrt{2} A B + 35 \right) I + 22 A + 75 \sqrt{2} B + 21 \sqrt{2} A B + 250} + 16 I + A + 2 \sqrt{2} B + \sqrt{2} A B + 23}{2 A + 2 \sqrt{2} B - 2}}}\right)},
$$
}
where $I = \frac{\lambda^2}{4}$, $A = A(I)=\sqrt{16 I + 25}$, and $B = B(I) = \sqrt{8 I - \sqrt{16 I + 25} + 5}$.
\end{coro}

These periodic points of period $2$ lead to the curve in Figure \ref{fig:logphilyapunovexponents2} that is labeled as ``Period  two''.

\section{Strict Inequalities Between Spectral Characteristics}\label{s.cohomology}

In this section we prove Theorem~\ref{t.strictinequalities}, that is, we establish the strict inequalities in \eqref{e.inequalities}. They will be a consequence of the following general result.

\begin{prop}\label{p.main1}
Suppose that $\sigma_A : \Sigma_A \to \Sigma_A$ is a topological Markov chain defined by a transitive $0-1$ matrix $A$ {\rm (}i.e., some power of $A$ has only positive entries{\rm )}, and $\phi : \Sigma_A \to \R$ is a H\"older continuous function. If $\phi$ is not cohomological to zero {\rm (}in other words, there are periodic orbits with different values of averages of $\phi$ over those orbits{\rm )}, then
\begin{multline}\label{e.strictin}
\inf_{p \in Per(f)} \left( \frac{1}{\pi(p)} \sum_{i=0}^{\pi(p)-1} \phi(f^i(p))\right) = \inf_{\mu \in \frak{M}} \int \phi \, d\mu < \int\phi\, d\mu_{max} < \\
< \sup_{\mu \in \frak{M}} \int \phi \, d\mu = \sup_{p \in Per(f)} \left( \frac{1}{\pi(p)} \sum_{i=0}^{\pi(p)-1} \phi(f^i(p)) \right),
\end{multline}
where $\mu_\mathrm{max}$ is the measure of maximal entropy, $\frak{M}$ is the space of all probability Borel $\sigma_A$-invariant measures, and $\pi(p)$ is the period of a periodic point $p$.
\end{prop}

\begin{proof}
First of all, due to Sigmund's Theorem \cite{Sig}, the ergodic measures supported on periodic orbits are (weak-*) dense in $\frak{M}$, which implies the equalities in \eqref{e.strictin}. In order to show the strict inequalities in \eqref{e.strictin}, we apply Proposition~\ref{p.thermodynamic} in the case where $\sigma_A : \Sigma_A \to \Sigma_A$ is conjugate to $T_\lambda|_{\Lambda_\lambda}$ and the potential is given by $\phi = -\log \|DT_\lambda|_{E^u}\|$.

Since by (5) the line $h_{\mathrm{top}} (\sigma_A) + t \int \phi \, d\mu_\mathrm{max}$ is tangent to the graph of $P(t\phi)$ at $(0, h_{\mathrm{top}} (\sigma_A))$, the strict convexity, which follows from (4), together with (6) implies Proposition~\ref{p.main1}.
\end{proof}

In order to prove Theorem~\ref{t.strictinequalities}, we will show that Proposition~\ref{p.main1} applies to the case at hand. This amounts to proving that there are periodic orbits in $\Lambda_\lambda$ with different values of the averaged unstable multipliers \cite[Proposition 20.3.10]{KH}. An averaged unstable multiplier of a periodic point $p$ of period $n$ is defined to be the $n$th root of the largest (in absolute value) eigenvalue of the differential $DT^n|_p$. Henceforth, we shall write simply \textit{multiplier} for \textit{averaged unstable multiplier}.

\begin{prop}\label{prop:cohom}
For every $\lambda > 0$, there exist two periodic points {\rm (}not necessarily of the same period{\rm )} in $\Lambda_\lambda$, such that their corresponding multipliers are distinct.
\end{prop}

This result is known for all $\lambda > 0$ sufficiently close to zero. Indeed, in that case one computes the multipliers for the period-six periodic point $p = (0, 0, a)$, with suitable $a\in \R$ such that $p\in S_0$, and for the fixed point $q = (1,1,1)$ explicitly. These multipliers are distinct, and a perturbation argument shows that for all $\lambda > 0$ sufficiently small, there exist a period-two periodic point and a period-six periodic point with different multipliers; see \cite{DG12} for details.

\begin{proof}[Proof of Proposition \ref{prop:cohom}.]
In \cite{Roberts1994} Baake and Roberts calculated a few periodic orbits, among which are two families of periodic points of period two and four, respectively. These are given, respectively, by
$$
P_a\eqdef \left(a, \frac{a}{2a - 1}, a\right)
$$
and
$$
Q_b \eqdef \left(-\frac{1}{2}, b, -\frac{1}{2}\right).
$$
Here, $a, b \in \R$. On each level surface $S_\lambda$, we can find points from these families. Namely, $a$ and $b$ simply need to be chosen in such a way that
\begin{equation}\label{e.ipaiqblambda}
I(P_a) = I(Q_b) = \frac{\lambda^2}{4},
\end{equation}
where $I$ is the Fricke-Vogt invariant. Note that $I(P_1) = I(Q_1) = 0$ and that $\lim_{a \to \infty} I(P_a) = \infty$ and $\lim_{b \to \infty} I(P_b) = \infty$. Thus, by continuity, for each $\lambda \ge 0$, it is possible to find $a,b \in [1,\infty)$ so that \eqref{e.ipaiqblambda} holds.

We claim that for every $\lambda \ge 0$, the multipliers of $P_a$ and $Q_b$ on $S_\lambda$ are different. The proposition obviously follows from this claim.

Assume that this claim fails. Then there exist $a,b \in [1,\infty)$ so that $I(P_a) = I(Q_b) \ge 0$ and the multipliers of $P_a$ and $Q_b$ coincide. The identity $I(P_a) = I(Q_b)$ implies that
\begin{equation}\label{e.multiplierproof1}
2a^2 + \frac{a^2}{(2a-1)^2} - \frac{2a^3}{2a-1} = \frac{1}{2} + b^2 - \frac{b}{2}.
\end{equation}
On the other hand, it was shown by Baake and Roberts that the unstable eigenvalue of $DT^2|_{P_a}$ is a root of the equation
\begin{equation}\label{e.multiplierproof3}
\mu^2 - \frac{8a^2 - 2a + 1}{2a-1} \mu + 1 = 0,
\end{equation}
while the unstable eigenvalue of of $DT^4|_{Q_b}$ is a root of the equation
\begin{equation}\label{e.multiplierproof4}
\mu^2 - (8(1 - 2b)b + 1) \mu + 1 = 0;
\end{equation}
see \cite[p.~850]{Roberts1994}. Due to Vieta's formulas, the roots of the equation $\mu^2 + (v^2 - 2) \mu + 1 = 0$ are squares of the roots of the equation $\mu^2 + v \mu + 1 = 0$. Thus, if the two multipliers in question coincide, it follows from \eqref{e.multiplierproof3} and \eqref{e.multiplierproof4} that
$$
\left( \frac{8a^2 - 2a + 1}{2a-1} \right)^2 - 2 = -8(1-2b)b-1,
$$
or, equivalently,
\begin{equation}\label{e.multiplierproof2}
\left( \frac{8a^2 - 2a + 1}{2a-1} \right)^2 - 2 = 16 \left( b^2 - \frac{b}{2} \right) - 1.
\end{equation}
It follows from \eqref{e.multiplierproof1} and \eqref{e.multiplierproof2} that
$$
\frac{(8a^2 - 2a + 1)^2}{(2a-1)^2} + 7 = 16 \frac{4a^4 - 6a^3 + 3a^2}{(2a-1)^2},
$$
which in turn implies that
$$
8a^3 - 4a + 1 = 0.
$$
Write $P(a) = 8a^3 - 4a + 1$. The critical numbers of this polynomial of degree $3$ are $\pm \frac{1}{\sqrt{6}}$. Thus, $P$ is strictly increasing on $[1,\infty)$. Since $P(1) = 5$, $P$ does not vanish for any $a \in [1,\infty)$; contradiction. This shows that the two multipliers cannot be equal, and the claim follows.
\end{proof}

\begin{rem}
The two families of periodic points of period two and four, respectively, used in the proof of Proposition \ref{prop:cohom} are the ones that lead to the curves in Figures~\ref{fig:logphilyapunovexponents} and \ref{fig:logphilyapunovexponents2} which are labeled with period two and four, respectively.
\end{rem}

\begin{proof}[Proof of Theorem \ref{t.strictinequalities}]
The strict inequalities $\gamma_\lambda < \dim_H \nu_\lambda <  \tilde \alpha^\pm_u(\lambda)$ follow directly from Theorem~\ref{t.identities}, Proposition~\ref{p.thermodynamic}, and Proposition~\ref{prop:cohom}. The inequalities  $\dim_H \nu_\lambda < \dim_H \Sigma_\lambda < \tilde \alpha^\pm_u(\lambda)$ follow from the strict convexity of the pressure function $P(t\phi)$ with $\phi = -\log \|DT_\lambda|_{E^u}\|$, the fact that $\dim_H \Sigma_\lambda$ is the only zero of $P(t\phi)$ (see \cite{MM}), and the expression for $\tilde \alpha^\pm_u(\lambda)$ from Theorem \ref{t.identities}.
\end{proof}

\section{Extensions and Generalizations}\label{s.extensions}

While we have focused up to this point on the classical Fibonacci Hamiltonian, much of what we do extends either partly or fully to other types of operators. Also, we strongly believe that the results presented here provide an insight toward and an opportunity to approach some other more complicated models as well. In this section we briefly address some of these extensions and generalizations.

\begin{itemize}
\item {\bf The Off-Diagonal Model.} In the present paper we consider the Fibonacci Hamiltonian in the form \eqref{e.fib}, which is usually called the {\it diagonal model} and which is the one most popular in the mathematics literature. In the physics literature the so-called off-diagonal model is usually considered. The spectral properties of the off-diagonal operator as well as the relation to the dynamics of the Fibonacci trace map are not any different from the diagonal one; see the appendix in \cite{DG11} for a detailed discussion of the off-diagonal model. All the results presented in this paper for the diagonal model also hold for the off-diagonal one.

\item {\bf Potentials Generated by Primitive Invertible Substitutions.} Discrete Schr\"odinger operators with potentials generated by primitive invertible substitutions have spectral properties that are very much similar to the spectral properties of the Fibonacci Hamiltonian. For some of the spectral properties this was justified in \cite{M14}. We expect that all the qualitative statements (i.e., all the statements mentioned in the introduction except Corollary~\ref{c.alphalowerbounds} and Theorem~\ref{t.asymptotics}) of this paper can be generalized to this case also. As for the large coupling asymptotics, the calculations can be more complicated, but can likely be carried out for particular potentials (given the results obtained in \cite{LPW07, LQW14, LW04, M10}).

\item {\bf Sturmian Potentials.} Sturmian potentials are natural generalizations of the Fibonacci potential. Namely, one simply replaces the specific value of $\alpha$ in \eqref{e.fib} by a general irrational $\alpha \in (0,1)$. It is known that the spectrum of a discrete Schr\"odinger operator with a Sturmian potential is a Cantor set of zero measure \cite{BIST89}, but in most cases this Cantor set will not be dynamically defined \cite{LPW07, LQW14}. Nevertheless, there is a dynamical presentation of the spectrum in this case as well \cite{BIST89, LW04, Ra}, and it would be interesting to see whether the dynamical approach can add something to the recent results in \cite{LPW07, LQW14, LW04, M10} that were obtained via the periodic approximation technique.

\item {\bf Jacobi Matrices.} In general, discrete Schr\"odinger operators form a particular case of the operators given by Jacobi matrices. In the case where the coefficients of a Jacobi matrix are modulated by the Fibonacci sequence, their spectral properties were studied in \cite{Yessen2011a}. Interestingly enough, the spectrum in this case does not have to be dynamically defined. Nevertheless, the relation to the dynamics of the Fibonacci trace map allows one to give a detailed description of the spectrum in this case, at least in some regimes, and our results can be used to provide a complete description throughout the entire parameter space.

\item {\bf CMV Matrices.} CMV matrices are the unitary analog of Jacobi matrices. That is, they are canonical models of unitary operators (just as Jacobi matrices are canonical models of self-adjoint operators) and they arise naturally in the study of orthogonal polynomials on the unit circle (while Jacobi matrices arise in the study of orthogonal polynomials on the real line); compare \cite{Simon1, Simon2}. In addition, CMV matrices have been effectively used to study quantum walks and the Ising model in one dimension; see \cite{CGMV, DMY13b}. Choosing the coefficients defining a CMV matrix according to the Fibonacci sequence one obtains an interesting model that can be studied using the trace map formalism as well. Several results for this model were obtained in \cite{DMY13a, DMY13b}, both from the perspective of orthogonal polynomials and the perspective of quantum walks and the Ising model. The results and tools developed in the present paper will allow one to take the analysis of the CMV case further.

\item {\bf Continuum Models.} Continuum Schr\"odinger operators with Fibonacci-type potentials were considered in \cite{BJK, DFG, KS, LSS}. In this case there are many models (depending on the choice of single-site potentials). The trace map description of the spectrum is also available in this case (see \cite{DFG}), and hence it is reasonable to expect that our results can be used.

\item {\bf Higher-Dimensional Separable Models.} Understanding the spectral properties of the operators associated with the standard two- and three-dimensional quasicrystal models is a major problem in the field which is currently out of reach. One of the ways to get some insight into the problem is to consider operators with separable potentials; for example, the Square (and Cubic) Fibonacci Hamiltonian and the labyrinth model \cite{EL06, EL07, EL08, Si89, SMS}. In these models the spectrum of the higher dimensional operator is the sum (or product) of the spectra of the one dimensional ones. Since studying the sum of dynamically defined Cantor sets is a classical problem which has been extensively studied (see, for example, \cite{HS, MY} and references therein), we expect that the current results will be instrumental in understanding the spectral properties of separable models. There is recent work on these models that relies on the one-dimensional results in the small and large coupling regimes \cite{DG11, DGS}, and the results of this paper will pave the way for a study of separable models that does not rely on the small and large coupling theory.

\end{itemize}

\section*{Acknowledgments}

We are grateful to Katrin Gelfert who explained to us how to use the thermodynamical formalism to prove Proposition~\ref{p.main1}, and to Eric Bedford, Tanya Firsova, Yulij Ilyashenko, and Mikhail Lyubich for useful discussions on polynomial dynamics.

\end{document}